\documentclass{amsart}
\usepackage{amssymb,stmaryrd,tikz}
\usepackage{hyperref}
\usepackage[margin=1.4in]{geometry}

\newtheorem{theorem}{Theorem}[section]
\newtheorem{corollary}[theorem]{Corollary}
\newtheorem{proposition}[theorem]{Proposition}

\newtheorem{lemma}[theorem]{Lemma}

\newtheorem{problem}[theorem]{Problem}
\theoremstyle{definition}
\newtheorem{definition}[theorem]{Definition}
\newtheorem{example}[theorem]{Example}
\newtheorem{remark}[theorem]{Remark}
\newtheorem{fact}[theorem]{Fact}

\newcommand{\bsig}{\ensuremath{\mbox{\boldmath${\Sigma}$}}}

\newcommand{\C}{\ensuremath{\mathfrak{C}}}
\newcommand{\cc}{\ensuremath{\mathfrak{c}}}
\newcommand{\m}{\ensuremath{\mathfrak{m}}}
\newcommand{\n}{\ensuremath{\mathfrak{n}}}

\newcommand{\Sc}{\ensuremath{\mathcal{C}}}
\newcommand{\Sf}{\ensuremath{\mathcal{F}}}

\newcommand{\fs}{\ensuremath{\mbox{\textup{\texttt{fs}}}}}
\newcommand{\cs}{\ensuremath{\mbox{\textup{\texttt{cs}}}}}
\newcommand{\crs}{\ensuremath{\mbox{\textup{\texttt{crs}}}}}
\newcommand{\lfs}{\ensuremath{\mbox{\textup{\texttt{lfs}}}}}
\newcommand{\lcs}{\ensuremath{\mbox{\textup{\texttt{lcs}}}}}
\newcommand{\gfs}{\ensuremath{\mbox{\textup{\texttt{gfs}}}}}
\newcommand{\gcs}{\ensuremath{\mbox{\textup{\texttt{gcs}}}}}

\newcommand{\Eq}[1]{\ensuremath{\llbracket #1\rrbracket}}

\newcommand{\dom}{\mbox{\textup{dom}}}
\newcommand{\ran}{\mbox{\textup{ran}}}
\newcommand{\res}{\restriction}
\newcommand{\st}{\ensuremath{\::\:}}

\newcommand{\selective}{selective}

\begin{document}

\title[Simultaneous Reducibility of Pairs of Borel Equivalence Relations]{Simultaneous Reducibility of Pairs of \\ Borel Equivalence Relations}

\author{Scott Schneider}
\address{Department of Mathematics, University of Michigan, 530 Church Street, Ann Arbor, MI 48109, U.S.A.}
\email{sschnei@umich.edu}

\begin{abstract}
Let $E\subseteq F$ and $E'\subseteq F'$ be Borel equivalence relations on the standard Borel spaces $X$ and $Y$, respectively. The pair $(E,F)$ is \emph{simultaneously Borel reducible} to the pair $(E',F')$ if there is a Borel function $f:X\to Y$ that is both a reduction from $E$ to $E'$ and a reduction from $F$ to $F'$. Simultaneous Borel embeddings and isomorphisms are defined analogously. We classify all pairs $E\subseteq F$ of smooth countable Borel equivalence relations up to simultaneous Borel bireducibility and biembeddability, and a significant portion of such pairs up to simultaneous Borel isomorphism. We generalize Mauldin's notion of \emph{Borel parametrization} \cite{MauldinBP} in order to identify large natural subclasses of pairs of smooth countable equivalence relations and of singleton smooth (not necessarily countable) equivalence relations for which the natural combinatorial isomorphism invariants are complete, and we present counterexamples outside these subclasses. Finally, we relate isomorphism of smooth equivalence relations and of pairs of smooth countable equivalence relations to Borel equivalence of Borel functions as discussed in Komisarski, Michalewski, and Milewski \cite{KMM}. \end{abstract}

\maketitle

\section{Introduction}\label{sec:intro}

Given equivalence relations $E$ and $F$ on sets $X$ and $Y$, a \emph{reduction} from $E$ to $F$ is a function $f:X\to Y$ such that $x\mathrel{E}y\Leftrightarrow f(x)\mathrel{F}f(y)$ for all $x,y\in X$. The study of reductions between equivalence relations in the Borel setting has been a significant area of research in descriptive set theory for at least the past twenty-five years. In this setting $X$ and $Y$ are \emph{Polish} (i.e., separable and completely metrizable) spaces, $E$ and $F$ are Borel subsets of the product spaces $X^2$ and $Y^2$, and $f$ is Borel measurable.

Recently there has been some attention focused on understanding pairs $E\subseteq F$ of definable equivalence relations. See, for instance, Miller \cite{BenThesis}, \cite{Ben}, Pinciroli \cite{Pin}, and Thomas \cite{STconvsbor} in the countable Borel setting, Motto Ros \cite{Luca} for pairs of analytic equivalence relations, and Feldman, Sutherland, Zimmer \cite{FSZ} and Danilenko \cite{Dan98}, \cite{Dan00} in the context of ergodic theory and orbit equivalence. Motivated by this interest, we define a \emph{simultaneous reduction} from the pair $E\subseteq F$ to the pair $E'\subseteq F'$ to be a function that is at once a reduction from $E$ to $E'$ and a reduction from $F$ to $F'$. Injective simultaneous reductions are called \emph{embeddings} and bijective ones \emph{isomorphisms}. The purpose of this paper is to initiate a study of certain classes of pairs of countable (i.e., having countable classes) Borel equivalence relations considered up to simultaneous Borel bireducibility and, to a lesser extent, biembeddability and isomorphism.

As the study of just singleton countable Borel equivalence relations under Borel reducibility is already quite difficult, a study of pairs will naturally begin with classes of equivalence relations for which the singletons are already well understood. The countable Borel equivalence relations that are best understood are the smooth and the hyperfinite ones. $E$ is \emph{smooth} if there is a Borel reduction from $E$ to equality of reals, and \emph{hyperfinite} if $E$ is the union of an increasing sequence of Borel equivalence relations with finite classes. It is easy (granting some deep results from the classical theory) to classify smooth countable equivalence relations up to Borel bireducibility, biembeddability, and isomorphism, and in \cite{DJK} Dougherty, Jackson, and Kechris classify hyperfinite equivalence relations up to these same notions of equivalence using results from ergodic theory.

In this paper we concentrate on the simplest class, and consider pairs $E\subseteq F$ of smooth countable equivalence relations. Perhaps surprisingly, classifying smooth countable pairs up to simultaneous Borel isomorphism turns out to be difficult and is closely related to the problem of classifying smooth (not necessarily countable) equivalence relations up to Borel isomorphism. We will identify combinatorial invariants for simultaneous Borel bireducibility, biembeddability, and isomorphism of arbitrary pairs of countable Borel equivalence relations, and prove that the first two of these invariants are complete for smooth countable pairs (Theorems \ref{thm:SmoothPairsRedSuf} and \ref{thm:SmoothPairsEmbSuf}). While we do not obtain a complete classification of smooth countable pairs up to simultaneous Borel isomorphism, we will be able to isolate a large natural subclass of such pairs on which our combinatorial invariant is complete (Theorem \ref{thm:SmoothPairsIsomBP}), and provide counterexamples outside this class (Theorem \ref{thm:SmoothPairsNotBP}). 
We also discuss the relative complexities of the isomorphism problems for smooth equivalence relations and pairs of smooth countable equivalence relations, and relate each to Borel equivalence of Borel functions as introduced in \cite{KMM}.

The study of simultaneous Borel isomorphism of smooth countable pairs leads us to generalize a notion from Mauldin \cite{MauldinBP} and define \emph{Borel parametrizations} of equivalence relations. We show that the class of smooth equivalence relations admitting a Borel parametrization is in some sense the largest natural subclass of Borel equivalence relations for which the obvious combinatorial isomorphism invariant is complete (Theorem \ref{thm:BP}). Relating to work of Komisarski, Michalewski, and Milewski \cite{KMM}, we identify a natural class of Borel functions strictly larger than the class of bimeasurable functions with the property that any two equivalent Borel functions from this class are in fact Borel equivalent to each other, and we argue that this is the largest meaningful class having this property.

The interest in sub-equivalence relations is partly motivated by the fact that some of the most important open problems concerning countable Borel equivalence relations involve sub-relations. These include, for instance, the problems associated with weakly universal countable Borel equivalence relations and the union problem for hyperfinite equivalence relations. The latter asks whether the union of an increasing sequence of hyperfinite equivalence relations is hyperfinite. Since singleton hyperfinite equivalence relations are already well-understood \cite{DJK} and the union problem involves chains $F_0\subseteq F_1\subseteq\cdots$ of hyperfinite equivalence relations, a natural next step in studying hyperfinite equivalence relations would be attempting to understand the manner in which one can lie inside another as a sub-relation. So the ideal class of pairs $E\subseteq F$ to study is the class of hyperfinite pairs, and it is the author's hope that such pairs will be studied from the perspective of simultaneous reducibility in future work.

We note that there has already been some work done on understanding hyperfinite pairs $E\subseteq F$ from a slightly different perspective than the one taken here. Building on \cite{BenThesis}, Miller \cite{Ben} and Pinciroli \cite{Pin} study equivalence relations on quotient spaces $X/E$ where $E$ is a countable Borel equivalence relation on the Polish space $X$. Here a set $B\subseteq X/E$ is \emph{Borel} if its lifting $\tilde{B}=\{x\in X\st [x]_E\in B\}$ is Borel, and a set $R\subseteq X/E\times Y/E'$ in a product of such quotients is \emph{Borel} if its lifting $\tilde{R}=\{(x,y)\st ([x]_E,[y]_{E'})\in R\}$ is Borel in $X\times Y$. Then a function $f:X/E\to Y/E'$ is \emph{Borel} if its graph is Borel, or equivalently if there is a Borel function $\tilde{f}:X\to Y$ (a \emph{lifting} of $f$) satisfying $\tilde{f}(x)\in f([x]_E)$ for all $x\in X$.  For countable Borel equivalence relations $E\subseteq F$ on $X$, the equivalence relation $F/E$ is defined on the quotient space $X/E$ by $[x]_E\mathrel{F/E}[y]_E\Leftrightarrow x\mathrel{F}y$. If $E'\subseteq F'$ is another such pair, $F/E$ and $F'/E'$ are \emph{isomorphic} if there is a bijective Borel reduction from $F/E$ to $F'/E'$. Miller \cite{Ben} classifies all finite Borel equivalence relations on the (unique up to isomorphism) hyperfinite quotient space $2^\omega/E_0$, and Pinciroli \cite{Pin} proves a version of the well-known Feldman-Moore representation theorem for a certain class of equivalence relations on quotient Borel spaces.

The relationship between Miller's framework and ours can be explained as follows.  Let $E\subseteq F$ and $E'\subseteq F'$ be countable Borel equivalence relations on $X$ and $Y$. If $f:X/E\to Y/F$ is an isomorphism from $F/E$ to $F'/E'$ in the sense of \cite{Ben}, then any lifting $\tilde{f}$ of $f$ is a simultaneous Borel reduction from the pair $(E,F)$ to the pair $(E',F')$ whose range meets every $E'$ class in $Y$. Conversely, any simultaneous Borel reduction $f$ from $(E,F)$ to $(E',F')$ induces an isomorphism between $F/E$ and the restriction of $F'/E'$ to the $E'$-saturation of $\ran(f)$. On the other hand, a simultaneous Borel isomorphism from $(E,F)$ to $(E',F')$ easily induces an isomorphism from $F/E$ to $F'/E'$, and in general simultaneous isomorphism is a stronger notion than isomorphism of quotients.

The remainder of this paper is organized as follows. In Section \ref{sec:prelim} we recall some notions from the theory of Borel equivalence relations and establish notation and terminology. In Section \ref{sec:singletons} we identify combinatorial invariants for Borel bireducibility, biembeddability, and isomorphism of Borel equivalence relations, recall that these first two invariants are complete for smooth equivalence relations, and identify Borel parametrized equivalence relations as a large natural class for which the third is complete. We then re-examine these results in the context of Borel equivalence of Borel functions as studied in \cite{KMM}. In Section \ref{sec:pairs0} we introduce combinatorial invariants for simultaneous Borel bireducibility, biembeddability, and isomorphism of pairs of countable Borel equivalence relations, and in Section \ref{sec:pairs1} we show that the first two of these invariants are complete for smooth countable pairs. Finally, in Section \ref{sec:pairs2} we again use Borel parametrizations in identifying a large natural subclass of smooth countable pairs for which the simultaneous isomorphism invariant is complete, and then we conclude with a discussion of the relative complexities of Borel equivalence of Borel functions, isomorphism of smooth equivalence relations, and simultaneous isomorphism of smooth countable pairs.

\smallskip

\noindent \textbf{Acknowledgements}. I would like to thank Andreas Blass and Rachel Basse for helpful discussions on the material in this paper.

\section{Preliminaries}\label{sec:prelim}

We will need a number of facts from descriptive set theory that will be well-known to experts but perhaps less well-known generally, so we include most of this material in Appendix \ref{sec:App1} at the end of the paper. In the present section we develop the background we will need concerning Borel equivalence relations, and establish some terminology and notation that is mostly (but not entirely) standard.

\subsection{Equivalence relations}

An equivalence relation $E$ is \emph{countable} if each $E$-class is countable, and \emph{finite} if each $E$-class is finite. If $E$ and $F$ are equivalence relations on sets $X$ and $Y$, a \emph{reduction} from $E$ to $F$ is a function $f:X\to Y$ such that $x\mathrel{E}y\Leftrightarrow f(x)\mathrel{F}f(y)$ for all $x,y\in X$. An injective reduction is called an \emph{embedding} and a bijective reduction an \emph{isomorphism}. $E$ is \emph{reducible} to $F$, written $E\leq F$, if there is a reduction from $E$ to $F$, \emph{embeddable} in $F$, written $E\sqsubseteq F$, if there is an embedding from $E$ to $F$, and \emph{isomorphic} to $F$, written $E\cong F$, if there is an isomorphism from $E$ to $F$.

A \emph{standard Borel space} is a measurable space $(X,\mathcal{B})$ for which there exists a Polish (i.e., separable and completely metrizable) topology on $X$ whose $\sigma$-algebra of Borel sets is $\mathcal{B}$. If $X$ is a standard Borel space, a set $A\subseteq X$ is \emph{analytic} if $A$ is the image of some standard Borel space under a Borel function, and \emph{coanalytic} if $X\setminus A$ is analytic. An equivalence relation $E$ on the standard Borel space $X$ is \emph{Borel} (\emph{analytic}, etc.) if $E$ is a Borel subset of $X\times X$. Throughout this paper we always work in the Borel setting. This means that equivalence relations live on standard Borel spaces and are Borel unless we explicitly mention otherwise, and all reductions are required to be Borel even if we fail to say so explicitly. It is customary to write a subscript ``$B$" in the Borel setting, and we use the following notation which is quite standard:
\[
\begin{array}{lll}
E\leq_BF        & \Longleftrightarrow & \mbox{$E$ is Borel reducible to $F$;} \\
E\sqsubseteq_BF & \Longleftrightarrow & \mbox{$E$ is Borel embeddable in $F$;} \\
E\cong_BF       & \Longleftrightarrow & \mbox{$E$ is Borel isomorphic to $F$;} \\
E\sim_BF        & \Longleftrightarrow & \mbox{$E$ is Borel \emph{bireducible} with $F$, i.e., $E\leq_BF\wedge F\leq_BE$;} \\
E\approx_BF     & \Longleftrightarrow & \mbox{$E$ is Borel \emph{biembeddable} with $F$, i.e., $E\sqsubseteq_BF\wedge F\sqsubseteq_BE$.}
\end{array}
\]

Let $E\subseteq F$ be equivalence relations on $X$ and let $A\subseteq X$. We write $[x]_E=\{y\in X\st x\mathrel{E}y\}$ for the $E$-equivalence class of $x\in X$, and $X/E=\{[x]_E\st x\in X\}$ for the quotient space of $E$-classes. We write $[A]_E=\{x\in X\st (\exists y\in A)\,x\mathrel{E}y\}$ for the $E$-saturation of $A$, and say that $A$ is $E$\emph{-invariant} if $[A]_E=A$. If $X$, $E$, and $A$ are Borel then $[A]_E$ is analytic in general and is Borel if $E$ is countable. We write $E\res A=\{(x,y)\in A^2\st x\mathrel{E}y\}$ for the restriction of $E$ to $A$, and $(E,F)\res A$ for the pair $(E\res A,F\res A)$. The equivalence relation $F/E$ on $X/E$ is defined by $[x]_E\mathrel{F/E}[y]_E\Leftrightarrow x\mathrel{F}y$. We will denote the equality relation on $X$ by $\Delta(X)$, and the indiscrete equivalence relation on $X$ by $I(X)=X^2$. If $G$ is an equivalence relation on $Y$, then $E\times G$ is the equivalence relation on $X\times Y$ defined by $(x,y)\mathrel{E\times G}(x',y')\Leftrightarrow x\mathrel{E}x'\wedge y\mathrel{G}y'$.

For a set $P\subseteq X\times Y$ in a product space we write $\pi_X(P)=\{x\in X\st (\exists y)(x,y)\in P\}$ for the projection of $P$ onto $X$, and $\pi_Y$ for the projection of $P$ onto $Y$. Sometimes instead we write $\pi_1$ ($\pi_2$) for the projection onto the first (second) factor, especially when $X=Y$. For $x\in X$ and $y\in Y$ we write $P_x$ and $P^y$ for the vertical and horizontal sections
\[
P_x=\{y\in Y\st (x,y)\in P\}, \quad P^y=\{x\in X\st (x,y)\in P\}.
\]
By the \emph{vertical section equivalence relation} on $P$ we mean the equivalence relation $E$ defined by $(x,y)\mathrel{E}(x',y')\Leftrightarrow x=x'$; the horizontal section equivalence relation is defined analogously.

\subsection{Smooth equivalence relations} Proofs of the claims in this section are standard, so we omit them. An equivalence relation satisfying the equivalent conditions of Proposition \ref{prop:smooth} is said to be \emph{smooth}. Note that every smooth equivalence relation is Borel.

\begin{proposition} Let $E$ be an equivalence relation on the standard Borel space $X$. Then the following are equivalent:
\begin{enumerate}
 \item there is a standard Borel space $Y$ such that $E\leq_B\Delta(Y)$;
 \item for every uncountable standard Borel space $Y$, $E\leq_B\Delta(Y)$;
 \item $E$ admits a countable Borel separating family, i.e., a family $(B_n)$ of $E$-invariant Borel subsets of $X$ such that for all $x,y\in X$, $x\mathrel{E}y$ if and only if $\forall n[x\in B_n\Leftrightarrow y\in B_n]$. \hfill $\square$
\end{enumerate}
\label{prop:smooth}
\end{proposition}

If $E$ is an equivalence relation on the standard Borel space $X$, the \emph{quotient} Borel structure on the set $X/E$ of $E$-classes is the largest $\sigma$-algebra making the canonical surjection $\pi_E:x\mapsto [x]_E$ measurable. Explicitly, $B\subseteq X/E$ is Borel in the quotient iff $\pi_E^{-1}(B)=\cup B$ is Borel in $X$. If $E$ is a Borel equivalence relation on $X$, the measurable space $X/E$ with its quotient Borel structure may or may not be standard Borel; if it is, then we will call $E$ \emph{standard}.

\begin{proposition} Let $E$ be an equivalence relation on the standard Borel space $X$. Then the following are equivalent:
\begin{enumerate}
 \item the quotient Borel structure on $X/E$ is standard Borel;
 \item there exists a standard Borel space $Y$ and a Borel reduction $f:X\rightarrow Y$ from $E$ to $\Delta(Y)$ such that $\ran(f)$ is Borel;
 \item $E$ is smooth and for any standard Borel space $Y$ and Borel reduction $f:X\rightarrow Y$ from $E$ to $\Delta(Y)$, $\ran(f)$ is Borel;
 \item for any standard Borel space $Y$ of cardinality $|X/E|$, there exists a surjective Borel reduction from $E$ to $\Delta(Y)$. \hfill $\square$
\end{enumerate}
\label{prop:standard}
\end{proposition}

If $E$ is an equivalence relation on $X$, a \emph{transversal} for $E$ is a subset $T\subseteq X$ that meets each $E$-class in exactly one point, and a \emph{selector} for $E$ is a function $\sigma:X\to X$ whose graph is contained in $E$ and whose image is a transversal.

\begin{proposition} Let $E$ be an equivalence relation on the standard Borel space $X$. Then the following are equivalent:
\begin{enumerate}
 \item $E$ admits a Borel transversal;
 \item $E$ admits a Borel selector. \hfill $\square$
\end{enumerate}
\label{prop:stronglysmooth}
\end{proposition}

There appears to be no uniform terminology in the literature either for the equivalent properties of Proposition \ref{prop:standard} or for the equivalent properties of Proposition \ref{prop:stronglysmooth}. As mentioned above, we will call equivalence relations satisfying the conditions of Proposition \ref{prop:standard} \emph{standard}. The term \emph{strongly smooth} appears in \cite{KL} for the property of admitting a Borel transversal, but it does not appear to be widely used and we suggest the term \emph{\selective{}} and use it throughout this paper. It is immediate from Proposition \ref{prop:standard} that every standard Borel equivalence relation is smooth. Of course, we also have the following:

\begin{proposition} Let $E$ be an equivalence relation on the standard Borel space $X$. If $E$ is \selective{} then $E$ is standard. \hfill $\square$ \end{proposition}

Thus \selective{} $\Rightarrow$ standard $\Rightarrow$ smooth. The next two basic examples show that these implications cannot be reversed (having all classes uncountable is not essential for this, but will be useful later).

\begin{example} Let $X$ be an uncountable standard Borel space and $P\subseteq X\times X$ a Borel set with all sections $P_x$ uncountable that does not admit a Borel uniformization (see \ref{ex:NotUniformizable}). Let $E$ be the vertical section equivalence relation on $P$. Then a Borel transversal for $E$ would be a Borel uniformization of $P$, so $E$ is not \selective{}. However, $[(x,y)]_E\mapsto x$ is a Borel isomorphism of $P/E$ with $X$, so $E$ is standard.\label{ex:nonselective} \end{example}

\begin{example} Let $X$ be an uncountable standard Borel space and $P\subseteq X\times X$ a Borel set with all nonempty sections $P_x$ uncountable and whose first projection is not Borel (see \ref{ex:AnalyticProjection}). Let $E$ be the vertical section equivalence relation on $P$. Then $X/E$ is Borel isomorphic to the first projection of $P$ and hence is not standard, even though $E$ is smooth as witnessed by the Borel reduction $(x,y)\mapsto x$ from $E$ to $\Delta(X)$.\label{ex:nonstandard} \end{example}

In Section \ref{sec:singletons} we will discuss equivalence relations admitting a \emph{Borel parametrization}, which is an even stronger property than admitting a Borel selector. But we point out that distinctions between these notions disappear in the context of \emph{countable} Borel equivalence relations.

\begin{fact} Let $E$ be a countable Borel equivalence relation on the standard Borel space $X$. Then the following are equivalent:
 \begin{enumerate}
  \item $E$ is smooth;
  \item $E$ is standard;
  \item $E$ is selective;
  \item $E$ is Borel parametrized (see Section \ref{sec:singletons});
  \item there exists a Borel partial order relation $\prec$ on $X$ such that the restriction of $\prec$ to each $E$-class has order type a subset of $\omega$.
  \item there is a sequence $T_n$ of Borel transversals for $E$ such that $\cup_nT_n=X$.
 \end{enumerate}
\label{fact:ctblSmooth}
\end{fact}

When (5) holds we say that $\prec$ orders each $E$-class in order type a subset of $\omega$ ``in a uniform Borel fashion."

\subsection{Split Borel equivalence relations}

We will use the following notation throughout. We write $\cc=2^{\aleph_0}$ and let
\[
\C \ := \ \omega\cup\{\aleph_0,\cc\} \quad \mbox{and} \quad \C^+ \ := \ \C\setminus\{0\},
\]
so that $\C^+$ is the set of possible cardinalities of nonempty Borel sets. We will generally try to use $i,j,k,l$ to range over $\omega$, $m,n$ to range over $\omega$ or occasionally $\omega\cup\{\aleph_0\}$, and $\m,\n$ to range over $\C$. If $\langle \kappa_i\st i\in I\rangle$ and $\langle \lambda_i\st i\in I\rangle$ are indexed families of cardinals with the same index set, we will write $\langle\kappa_i\rangle\leq\langle\lambda_i\rangle$ to mean $\kappa_i\leq\lambda_i$ for all $i\in I$. Given the Borel equivalence relation $E$ on the standard Borel space $X$, for each $\m\in\C^+$ let
\[
X^{(E)}_{\m} \ := \ \{x\in X\::\: |[x]_E|=\m\},
\]
and define $X^{(E)}_{\leq\m}$, $X^{(E)}_{<\m}$, etc., in the obvious manner. If $E$ is clear from context we may just write $X_{\m}$ instead of $X^{(E)}_{\m}$.

Now we need a technical notion for which there is no standard term in the literature. If $E$ is a Borel equivalence relation on $X$, then $X^{(E)}_{\cc}$ is analytic and each $X^{(E)}_{n}$, $1\leq n\leq\omega$, is coanalytic (see \ref{fact:AnalyticRep} and \ref{fact:genunicity}). We will say that $E$ \emph{splits} if $X^{(E)}_{\cc}$ is Borel. If $E$ splits, then in fact $X^{(E)}_{\m}$ is Borel for each $\m\in\C^+$ (see \ref{prop:card}). A split Borel equivalence relation need not be smooth, and a smooth split equivalence relation need not be standard, as witnessed by Example \ref{ex:nonstandard}. Similarly, a standard split equivalence relation need not be \selective{}, as witnessed by Example \ref{ex:nonselective}. On the other hand, it is also easy to construct a smooth (even \selective{}) equivalence relation that does not split.

\begin{example} Let $X$ and $Y$ be uncountable standard Borel spaces, and let $C\subseteq X\times Y$ be a Borel set with each section $C_x$ nonempty and countable, so that in particular $C$ admits a Borel uniformization $T\subseteq C$. Let $A\subseteq X$ be a non-Borel analytic set, and using \ref{fact:AnalyticRep} let $U\subseteq X\times Y$ be a Borel set with each nonempty section $U_x$ uncountable whose projection is $A$. Let $B=C\cup U$, and let $E$ be the vertical section equivalence relation on $B$. Then $T$ is a Borel transversal for $E$, so $E$ is selective. However $E$ does not split, since if $B_{\leq\omega}^{(E)}$ were Borel then $\pi_X\big(B_{\leq\omega}^{(E)}\big)=X\setminus A$ would be Borel. \end{example}

Thus we have the following diagram of implications, none of which can be reversed (see Section \ref{sec:singletons} for the definition of a Borel parametrized equivalence relation).

\begin{figure}[h]
\[
\begin{array}{ccccccc}
\mbox{smooth and countable} & \Rightarrow & \mbox{Borel parametrized} & \Rightarrow & \mbox{split} & & \\
& & \Downarrow & &  & & \\
& & \mbox{\selective{}} & \Rightarrow & \mbox{standard} & \Rightarrow & \mbox{smooth}
\end{array}
\]
\caption{A diagram of implications between properties of equivalence relations.}
\end{figure}

\section{Smooth equivalence relations and Borel parametrizations}\label{sec:singletons}

We begin by identifying the natural combinatorial invariants for $\sim_B$, $\approx_B$, and $\cong_B$ of Borel equivalence relations.

\begin{definition}\label{df:shape}
Suppose $E$ is a Borel equivalence relation on the standard Borel space $X$. For each cardinal $\m\in\C^+$, let $\n_{\m}(E)$ be the number of $E$-classes of size $\m$, let $\n_{\geq\m}(E)$ be the number of $E$-classes of size at least $\m$, and let
\[
\n(E) := |X/E| = \sum_{\m\in\C^+}\n_{\m}(E)
\]
be the total number of $E$-equivalence classes. Define the \emph{fine shape} of $E$ to be the sequence
\[
\fs(E) := \langle\n_{\m}(E)\st \m\in\C^+\rangle
\]
and the \emph{coarse shape} of $E$ to be the sequence
\[
\cs(E) := \langle \n_{\geq\m}(E)\st \m\in\C^+\rangle.
\]
\end{definition}

In the appendix we show that $\n_\cc(E)\in\C$, $\n_{\m}(E)\in\C\cup\{\aleph_1\}$ for all $\m\in\C^+$, and if $E$ splits then in fact $\n_{\m}(E)\in\C$ for all $\m\in\C^+$ (see \ref{prop:card}). Of course, under AD (or CH) we will have each $\n_{\m}(E)\in\C$ regardless, but the above restrictions on fine shape are the only ones that are provable in ZFC. Indeed, by \ref{prop:pathological} it is consistent that for every function $\alpha:\C^+\to\C\cup\{\aleph_1\}$ such that $\alpha(\cc)=\cc$, there is a smooth equivalence relation $E$ with $\fs(E)=\alpha$.

It is a trivial exercise to check that the combinatorial notions of Definition \ref{df:shape} provide necessary conditions for the existence of Borel reductions, embeddings, and isomorphisms between Borel equivalence relations.

\begin{proposition}\label{prop:shape} Let $E$ and $F$ be Borel equivalence relations.
\begin{enumerate}
 \item[(i)] If $E\leq_BF$, then $\n(E)\leq\n(F)$.
 \item[(ii)] If $E\sqsubseteq_BF$, then $\cs(E)\leq \cs(F)$.
 \item[(iii)] If $E\cong_BF$, then $\fs(E)=\fs(F)$. \hfill $\square$
\end{enumerate}
\end{proposition}

We now discuss the extent to which (i), (ii), and (iii) of Proposition \ref{prop:shape} admit converses. Of course, the largest class of Borel equivalence relations for which the converses could conceivably hold is the class of smooth equivalence relations. It is well-known, and an easy consequence of Silver's theorem (\ref{fact:Silver}), that $\n(E)$ is indeed a complete invariant for Borel bireducibility of smooth equivalence relations.

\begin{proposition}\label{prop:smoothred} Let $E$ and $F$ be smooth equivalence relations on the standard Borel spaces $X$ and $Y$, respectively. Then $E\leq_BF$ if and only if $\n(E)\leq\n(F)$. In particular, $E\sim_BF$ if and only if $\n(E)=\n(F)$. \hfill $\square$ \end{proposition}

It is perhaps less well-known that $\cs(E)$ is a complete invariant for Borel biembeddability of smooth equivalence relations.

\begin{proposition}\label{prop:smoothemb} Let $E$ and $F$ be smooth equivalence relations on the standard Borel spaces $X$ and $Y$, respectively. Then $E\sqsubseteq_BF$ if and only if $\cs(E)\leq\cs(F)$. In particular, $E\approx_BF$ if and only if $\cs(E)=\cs(F)$. \end{proposition}

\begin{proof} The forward direction holds for any Borel equivalence relations by Proposition \ref{prop:shape}. For the converse, we consider several cases.

\underline{Case 1}: $X$ is countable. In this case any map defined on $X$ will be Borel, so we may freely use the axiom of choice. Well-order $X$ as $X=\{x_n \st n<|X|\}$, and inductively define $f(x_n)$ as follows. If there is no $m<n$ such that $x_m\mathrel{E}x_n$, let $\m\geq |[x_n]_E|$ be least for which there exists an $F$-class of size $\m$ disjoint from $\{f(x_i) \st i<n\}$, and define $f(x_n)$ to be any element in any such $F$-class; otherwise, if $m<n$ is such that $x_m\mathrel{E}x_n$, let $f(x_n)$ be any element in $[f(x_m)]_F\setminus\{f(x_i) \st i<n\}$. Such choices always exist since $\cs(E)\leq\cs(F)$, and the function $f$ thus defined is an embedding.

\underline{Case 2}: $\n_\cc(F)\leq\aleph_0$, so that in particular $X_\m$ and $Y_\m$ are Borel for each $\m\in\C^+$. In this case let
\[
\{[x_k]_E \st k<\n_\cc(E)\} \quad \mbox{and} \quad \{[y_k]_F \st k<\n_\cc(F)\}
\]
enumerate the $E$-classes and $F$-classes in $X_\cc$ and $Y_\cc$ respectively, noting that $\n_\cc(E)\leq\n_\cc(F)$. If $\n_\cc(F)$ is finite, then for each $k<\n_\cc(E)$ let $f_k$ be a Borel isomorphism from $[x_k]_E$ to $[y_k]_F$, and if $\n_\cc(F)=\aleph_0$, then for each $k<\n_\cc(E)$ let $f_k$ be a Borel isomorphism from $[x_k]_E$ to $[y_{2k}]_F$. Then $f=\cup f_k$ is a Borel embedding from $E\res X_\cc$ to $F\res Y_\cc$ such that 
\[
\cs(E\res X_{\leq\omega}) \ \leq \ \cs(F\res Y\setminus [\ran(f)]_F).
\]
Hence without loss of generality we may assume that $\n_\cc(E)=0$, so that $E$ is a smooth countable equivalence relation. For the rest of Case 2 we make this assumption.

\underline{Case 2a}: there is largest $m\leq\omega$ such that $\n_m(F)$ is uncountable. Fix such $m$, and fix also a Borel reduction $f:X_{\leq m}\to Y_m$ from $E\res X_{\leq m}$ to $F\res Y_m$. Using \ref{fact:ctblSmooth}, well-order each $E$-class in $X_{\leq m}$ and each $F$-class in $Y_m$ in order type $\omega$ or finite in a uniform Borel manner, and then define an injective Borel reduction $g:X_{\leq m}\rightarrow Y_m$ from $E\res X_{\leq m}$ to $F\res Y_m$ by sending the $n$th element of the $E$-class $[x]_E$ to the $n$th element of the $F$-class $[f(x)]_F$. Note that $\cs(E\res X_{>m})\leq\cs(F\res Y\setminus [\ran(g)]_F)$ and that $X_{>m}$ is countable, so now we are finished by Case 1.

\underline{Case 2b}: there is no largest $m\leq\omega$ such that $\n_m(F)$ is uncountable. If $\n(F)$ is countable then we are back in Case 1, so we may assume that there exist finite $m_0<m_1<\cdots$ such that for each $i$, $\n_{m_i}(F)$ is uncountable. For each $i$ let
\[
Z_i \ = \ X_{\leq m_i}\setminus \left(\bigcup_{j<i}X_{\leq m_j}\right).
\]
Then for each $i$ there is a Borel embedding $g_i$ from $E\res Z_i$ to $F\res Y_{m_i}$ by the argument given in Case 2a. Then $\cup_ig_i$ is a Borel embedding from $E\res X_{<\omega}$ to $F\res Y_{<\omega}$, and by assumption $\n_\omega(E)\leq\n_{\geq\omega}(F)\leq\aleph_0$, so $E\res X_\omega\sqsubseteq_B F\res Y_{\geq\omega}$ by Case 1. This completes Case 2.

\underline{Case 3}: $\n_\cc(F)=\cc$. Using the fact that $F$ is smooth, fix a Borel reduction $f:Y\rightarrow\mathcal{N}$ from $F$ to $\Delta(\mathcal{N})$. Then $\mbox{graph}(f)$ is a Borel subset of $Y\times\mathcal{N}$ such that the set
\[
U:=\{\alpha\in\mathcal{N}\st f^{-1}(\{\alpha\})\mbox{ is uncountable}\}
\]
is uncountable. Hence by Fact \ref{prop:Mauldin} there is a nonempty perfect set $P\subseteq U$ and a Borel isomorphism $\phi$ of $2^\omega\times P$ onto a subset $R$ of $\mbox{graph}(f)$ such that for each $p\in P$, $\phi\res 2^\omega\times\{p\}$ is a Borel isomorphism onto $R^p\times\{p\}$. Fix a compatible Polish topology on $X$ with countable base $(B_n)$. Using the fact that $E$ is smooth, let $g:X\rightarrow P$ be a Borel reduction from $E$ to $\Delta(P)$, and define the Borel embedding $g':X\rightarrow 2^\omega\times P$ from $E$ into $I(2^\omega)\times\Delta(P)$ by $g'(x)=(\alpha_x,g(x))$, where $\alpha_x(n)=1\Leftrightarrow x\in B_n$. Now $\pi_Y\circ\phi\circ g':X\rightarrow Y$ is a Borel embedding from $E$ into $F$. \end{proof}

For $\cong_B$, we have the following partial converse of Proposition \ref{prop:shape}(iii) which will later be generalized in Theorem \ref{thm:BP}.

\begin{proposition} Let $E$ and $F$ be smooth equivalence relations with at most countably many uncountable equivalence classes. Then $E\cong_BF$ if and only if $\fs(E)=\fs(F)$. \label{prop:smoothisom} \end{proposition}

\begin{proof} The forward direction follows from Proposition \ref{prop:shape}. For the converse, first decompose $X$ and $Y$ into the Borel sets
\[
X=\bigsqcup_{\m\in\C^+}X_{\m}\quad\mbox{and}\quad Y=\bigsqcup_{\m\in\C^+}Y_{\m}.
 \]
Here $X_\cc$ and $Y_\cc$ are Borel because $\n_\cc(E)=\n_\cc(F)\leq\aleph_0$, and consequently each $X_{\m}$, $Y_{\m}$ is Borel by \ref{prop:card}. Now we can write
\[
X_\cc=\bigsqcup_{k<\n^E_\cc}[x_k]_E\quad\mbox{and}\quad Y_\cc=\bigsqcup_{k<\n^F_\cc}[y_k]_F
 \]
and use the isomorphism theorem (\ref{fact:kuratowski}) to obtain Borel isomorphisms $f_k:[x_k]_E\rightarrow [y_k]_F$, $k<\n_\cc(E)=\n_\cc(F)$. Then $\cup_kf_k$ is a Borel isomorphism of $E\res X_\cc$ with $F\res Y_\cc$, and we may turn our attention to $X_{\leq\omega}$ and $Y_{\leq\omega}$.

Fix $m\leq\omega$. $E\res X_m$ and $F\res Y_m$ are smooth countable equivalence relations, so $X_m/E$ and $Y_m/F$ are standard Borel spaces. Since
\[
|X_m/E| \ = \ \n_m(E) \ = \ \n_m(F) \ = \ |Y_m/F|,
\]
we may fix a Borel isomorphism $\tilde{g}_m$ from $X_m/E$ to $Y_m/F$. Using \ref{fact:ctblSmooth}, order each $E$-class in $X_m$ and each $F$-class in $Y_m$ in order type $m$ in a uniform Borel manner, and for each $x\in X_m$ let $k(x)<m$ be the position of $x$ within $[x]_E$ under this ordering. Define the function $g_m:X_m\to Y_m$ by using $\tilde{g}_m$ and matching up the orderings, so that $g_m(x)$ is the $k(x)$th element in the $F$-class $\tilde{g}_m([x]_E)$. Then $g_m$ is a Borel isomorphism from $E\res X_m$ to $F\res Y_m$, and $\bigcup_mg_m$ is a Borel isomorphism from $E\res X_{\leq\omega}$ to $F\res Y_{\leq\omega}$, completing the proof. \end{proof}

One might hope to extend Proposition \ref{prop:smoothisom} to arbitrary smooth equivalence relations, but we can easily observe that this is impossible. (Note that Proposition \ref{prop:smoothisom} and Example \ref{ex:FSnotcomplete} may be viewed as analogues, respectively, of $(ii)\Rightarrow (i)$ and $(i)\Rightarrow (ii)$ of \cite[Proposition 11]{KMM}).

\begin{example} There exist smooth equivalence relations $E$ and $F$ such that $\fs(E)=\fs(F)$ but $E\not\cong_BF$.\label{ex:FSnotcomplete} \end{example}

\begin{proof} Let $E$ be a smooth but not selective equivalence relation with uncountably many equivalence classes all of which are uncountable, as in Example \ref{ex:nonselective}. Let $F$ be the vertical section equivalence relation on $\mathcal{N}\times\mathcal{N}$. Then $\fs(E)=\fs(F)$ but $E\not\cong_BF$, since $F$ is \selective{} and $E$ is not. \end{proof}

We will see below that there exists a non-isomorphic pair of \selective{} equivalence relations all of whose uncountably many equivalence classes are uncountable, so that fine shape fails to be a complete invariant even in this specialized case. First we characterize a natural class of smooth equivalence relations for which fine shape \emph{is} a complete isomorphism invariant.

\begin{definition}\label{df:BP} Let $E$ be an equivalence relation on the standard Borel space $X$. A \emph{Borel parametrization} of $E$ is a Borel bijection
\[
\phi:X\to\bigsqcup_{\m\in\C^+}\left(Z_{\m}\times Y_{\m}\right),\quad\mbox{where}
\]
\begin{enumerate}
 \item[(i)] $(Z_\m \st \m\in\C^+)$ is a pairwise disjoint sequence of (possibly empty) standard Borel spaces;
 \item[(ii)] $(Y_{\m} \st \m\in\C^+)$ is a sequence of standard Borel spaces such that for each $\m\in\C^+$, $|Y_{\m}|=\m$; and
 \item[(iii)] for all $x,y\in X$, $x\mathrel{E}y\,\Leftrightarrow\,(\pi_1\circ\phi)(x)=(\pi_1\circ\phi)(y)$.
\end{enumerate}
We say that $E$ is \emph{Borel parametrized} if $E$ admits a Borel parametrization.
\end{definition}

It is easy to see that Borel parametrized equivalence relations with the same fine shape are Borel isomorphic to each other. Indeed, let $Y_\cc=2^\omega\times\{\cc\}$ and for each $1\leq\m\leq\omega$ let $Y_{\m}=\m\times\{\m\}$, so that $Y_{\m}$ is an explicit Polish space of cardinality $\m$ for each $\m\in\C^+$, and then for any $\alpha:\C^+\to\C$ let $\mathcal{E}(\alpha)$ be the equivalence relation
\[
\mathcal{E}(\alpha)\::=\:\bigsqcup_{\m\in\C^+}\left(\Delta(Y_{\alpha(\m)})\times I(Y_{\m})\right)
\]
defined on the standard Borel space
\[
Y(\alpha)\::=\:\bigsqcup_{\m\in\C^+}\left(Y_{\alpha(\m)}\times Y_{\m}\right).
\]
We make the following observations:
\begin{enumerate}
 \item for each $\alpha\in\C^{\C^+}$, $\mathcal{E}(\alpha)$ is a split, \selective{} Borel equivalence relation on $Y(\alpha)$;
 \item if $E$ is Borel parametrized, then $E\cong_B\mathcal{E}(\fs(E))$.
\end{enumerate}
In particular, every Borel parametrized equivalence relation is Borel, \selective{}, and splits, and if $E$ is Borel parametrized then we can think of $\mathcal{E}(\fs(E))$ as a convenient normal form representation of $E$.

\begin{theorem}\label{thm:BP}
If $E$ and $F$ are Borel parametrized equivalence relations, then $E\cong_BF$ if and only if $\fs(E)=\fs(F)$. Moreover, if $E$ is a smooth equivalence relation that is not Borel parametrized, then there exists a smooth equivalence relation $F$ such that $\fs(E)=\fs(F)$ but $E\not\cong_BF$.
\end{theorem}

\begin{proof} The first claim follows immediately from Proposition \ref{prop:shape} and observation (2) above.

For the second claim, suppose that $E$ is a smooth equivalence relation on the standard Borel space $X$ that is not Borel parametrized. If $E$ splits, then $\fs(E)$ is the fine shape of some Borel parametrized equivalence relation (namely $\mathcal{E}(\fs(E))$) which is not isomorphic to $E$. Hence we may assume that $E$ does not split, so in particular $\n_{\cc}(E)=\cc$. We will now complete the proof by constructing a pair of smooth equivalence relations $F_1$ and $F_2$ with the same fine shape as $E$ such that $F_2$ is standard and $F_1$ is not.

Fix a Borel reduction $f:X\to 2^\omega$ from $E$ to $\Delta(2^\omega)$. Let $P$ be a Borel subset of $2^\omega\times 2^\omega$ such that every nonempty section $P^\alpha$ of $P$ is uncountable but $\pi_2(P)$ is not Borel. Let
\[
Y\:=\:\mbox{graph}(f)\oplus P\:\subseteq\: (X\times 2^\omega)\oplus (2^\omega\times 2^\omega),
\]
and let $F_1$ be the horizontal section equivalence relation on $Y$ defined by $(x,\alpha)\mathrel{F_1}(x',\alpha')$ iff $\alpha=\alpha'$ and either $(x,\alpha)$, $(x',\alpha')$ both lie in $\mbox{graph}(f)$ or they both lie in $P$. Then $F_1$ has the same fine shape as $E$ and $F_1$ is not standard. On the other hand, by \ref{fact:genunicity} the set
\[
C=\{y\in 2^\omega \st f^{-1}(y)\mbox{ is countable and nonempty}\}
\]
is coanalytic. Let $B\subseteq X\times 2^\omega$ be a Borel set each of whose sections $B^\alpha$ is uncountable and whose second projection is $2^\omega\setminus C$. Let $F_2$ be the horizontal section equivalence relation on $B\cup\mbox{graph}(f)$. Then $F_2$ has the same fine shape as $E$, and $F_2$ is standard. \end{proof}

In light of this result, we can rephrase Proposition \ref{prop:smoothisom} as follows:

\begin{corollary}\label{cor:ctblBP} If $E$ is a smooth equivalence relation with at most countably many uncountable equivalence classes, then $E$ is Borel parametrized. In particular, any smooth countable equivalence relation is Borel parametrized. \hfill $\square$ \end{corollary}

Next we present an example due to Mauldin which shows that even within the class of \selective{} equivalence relations each of whose uncountably many equivalence classes is uncountable, fine shape is not a complete isomorphism invariant.

\begin{example} There exists a \selective{} equivalence relation $E$ each of whose uncountably many equivalence classes is uncountable such that $E$ is not Borel parametrized. \label{ex:MauldinBP} \end{example}

\begin{proof} By Mauldin \cite[3.2]{MauldinBP}, there is a closed subset $B\subseteq [0,1]\times [0,1]$ such that:
\begin{enumerate}
 \item[(i)] for each $x\in [0,1]$, $B_x$ is uncountable (and closed);
 \item[(ii)] there is no Borel isomorphism $g:[0,1]\times [0,1]\to B$ such that for each $x\in X$, the function $g(x,\cdot)$ maps $[0,1]$ onto $B_x$.
\end{enumerate}
Taking $B$ to realize this example, let $E$ be the vertical section equivalence relation on $B$. Since each $B_x$ is compact, $B$ admits a Borel uniformization by a classical result of Novikov (see for instance \cite[28.8]{Kechris}), and hence $E$ is \selective{}. Suppose now for contradiction that $E$ is Borel parametrized. Then in fact there is a Borel parametrization $\phi$ of $E$ taking values in $[0,1]\times [0,1]$. Given such a parametrization $\phi$, define the Borel automorphism $\sigma$ of $[0,1]$ by $\sigma(x)=(\pi_1\circ\phi^{-1})(x,0)$. Then the function $g:[0,1]\times [0,1]\to B$ defined by $$g(x,y)=(\, x,\;(\,\pi_2\circ\phi^{-1})(\sigma^{-1}(x),y)\,)$$ is a Borel isomorphism such that $g(x,\cdot)$ maps $[0,1]$ onto $B_x$, contradicting (ii). \end{proof}

\begin{remark} In \cite{MauldinBP}, Mauldin calls a Borel set $B$ in the product $X\times Y$ of Polish spaces \emph{Borel parametrized} if there exists a Borel set $Z\subseteq Y$ and a Borel isomorphism $g:X\times Z\to B$ such that for each $x\in X$, $g(x,\cdot)$ maps $Z$ onto $B_x$. The argument given in Example \ref{ex:MauldinBP} shows that if all nonempty sections of $B\subseteq X\times Y$ have the same cardinality, then $B$ is Borel parametrized in the sense of \cite{MauldinBP} if and only if the vertical section equivalence relation on $B$ is Borel parametrized in our sense. In light of this we consider our definition to be a natural generalization of Mauldin's. \end{remark}

Theorem \ref{thm:BP} leaves us with the following:

\begin{problem} Classify smooth equivalence relations (that are not Borel parametrized) up to Borel isomorphism. \label{prob:1} \end{problem}

\begin{remark}\label{rem:BS} Let $\mathcal{B}$ be the class of Borel parametrized equivalence relations. Theorem \ref{thm:BP} shows that fine shape is a complete isomorphism invariant on $\mathcal{B}$, and we argue that $\mathcal{B}$ is in some sense the largest natural class of Borel equivalence relations with this property. Given any fine shape function $\alpha$, if $\aleph_1$ is a value of $\alpha$ then there is no canonical way of recovering a smooth equivalence relation $E$ with $\fs(E)=\alpha$, and if $\aleph_1$ is not a value of $E$ then the most natural smooth equivalence relation $E$ with fine shape $\alpha$ is the Borel parametrized equivalence relation $\mathcal{E}(\alpha)$. Hence if $\mathcal{U}\not\subseteq\mathcal{B}$ is any collection of smooth equivalence relations on which fine shape is a complete isomorphism invariant, then $\mathcal{U}$ must contain some ``unnatural" smooth equivalence relation that cannot be recovered canonically from its fine shape. \end{remark}

In \cite{KMM}, Komisarski, Michalewski, and Milewski study \emph{equivalence} and \emph{Borel equivalence} of functions between Polish spaces, the latter of which is closely related to Borel isomorphism of smooth equivalence relations. Here functions $f,g:X\to Y$ between Polish spaces $X$ and $Y$ are \emph{equivalent} if there is a bijection $\phi:X\to X$ such that $f=g\phi$, and \emph{Borel equivalent} if there is a Borel such $\phi$. For convenience we introduce the following notation, which will appear again in Section \ref{sec:pairs2}. Given standard Borel spaces $X$ and $Y$ and a Borel function $g:X\to Y$, let $\Eq{g}$ denote the smooth equivalence relation on $X$ defined by
\[
x\mathrel{\Eq{g}}y \ \Longleftrightarrow \ g(x)=g(y).
\]
If Borel functions $f,g:X\to Y$ are Borel equivalent with witness $\phi$, then $\phi$ is a Borel isomorphism from $\Eq{f}$ to $\Eq{g}$.

The authors of \cite{KMM} prove that the Borel function $f:X\to Y$ is bimeasurable if and only if every Borel function $g:X\to Y$ that is equivalent to $f$ is in fact Borel equivalent to $f$ (\cite[Proposition 11]{KMM}). Recalling Purves' theorem \cite{Purves} that a Borel function $f:X\to Y$ is bimeasurable if and only if $\{y\in Y\st f^{-1}(\{y\})\mbox{ is uncountable}\}$ is countable, we see that the analogous result in our context is that if $E$ is smooth, then $\fs(E)$ determines $E$ up to Borel isomorphism within the class of smooth equivalence relations if and only if all but countably many $E$-classes are countable. Using the notion of Borel parametrization, we can state the analogue of our Theorem \ref{thm:BP} in the context of Borel equivalence of Borel functions as follows.

\begin{proposition}\label{prop:KMM} Let $X$, $Y$ be Polish spaces and $f,g:X\to Y$ Borel functions such that $\Eq{f}$ and $\Eq{g}$ are Borel parametrized. Then $f$ and $g$ are equivalent if and only if they are Borel equivalent. Moreover, if $h:X\to Y$ is a Borel function such that $\Eq{h}$ is not Borel parametrized, then there is Borel $h':X\to Y$ such that $h$ and $h'$ are equivalent but not Borel equivalent. \end{proposition}

\begin{proof} Let $f,g:X\to Y$ be Borel with $\Eq{f}$ and $\Eq{g}$ Borel parametrized, and suppose $f$ and $g$ are equivalent. Let
\[
\phi_f:X\to\bigsqcup_{\m\in\C^+}\left(Z_\m\times Y_\m\right) \quad \mbox{and} \quad \phi_g:X\to\bigsqcup_{\m\in\C^+}\left(Z_\m'\times Y_\m\right)
\]
be Borel parametrizations of $\Eq{f}$ and $\Eq{g}$, respectively. Let $A_f\subseteq\C^+$ be the set of $\m\in\C^+$ such that $Z_{\m}\ne\emptyset$, or equivalently the set of $\m\in\C^+$ for which $f$ has some fiber of cardinality $\m$. Define $A_g\subseteq\C^+$ analogously. Since $f$ and $g$ are equivalent we have $A_f=A_g$, so we write simply $A=A_f=A_g$.  Now let $\m\in A$ be arbitrary.  Fix $y_\m\in Y_\m$, and using the fact that $f,g$ are equivalent define the bijection $\psi_\m:Z_\m\to Z_{\m}'$ so that for each $z\in Z_\m$, $\psi_{\m}(z)$ is the unique element of $Z_{\m}'$ such that
\[
f\bigg(\phi_f^{-1}(z,y_\m)\bigg) \ = \ g\bigg(\phi_g^{-1}(\psi_{\m}(z),y_\m)\bigg).
\]
Then $\psi_\m$ is Borel, and hence is an isomorphism. Write $\psi_\m\times\mbox{id}_\m$ for the map $(z,y)\mapsto (\psi_\m(z),y)$  from $Z_\m\times Y_\m$ to $Z_{\m}'\times Y_\m$. Letting $\m$ vary over $A$, write $\displaystyle{\psi=\bigsqcup_{\m\in A}(\psi_\m\times\mbox{id}_\m})$. Then
\[
\phi \ := \ \phi_g^{-1}\circ\psi\circ\phi_f
\]
is a Borel automorphism of $X$ such that $f=g\phi$. This proves the first claim. In light of Purves' theorem, the second follows immediately from Corollary \ref{cor:ctblBP} and \cite[Proposition 11]{KMM}. \end{proof}

Moreover, Proposition \ref{prop:KMM} is optimal in the sense of Remark \ref{rem:BS}; that is, the class of all Borel functions $f$ for which $\Eq{f}$ is Borel parametrized is the largest class of Borel functions whose Borel equivalence type can be recovered canonically from its equivalence type alone.

We conclude this discussion by observing that Borel equivalence of Borel functions is too strong to be an exact analogue of Borel isomorphism of smooth equivalence relations. Let us say that Borel functions $f,g:X\to Y$ are \emph{weakly equivalent} if there exist bijections $\phi:X\to X$ and $\psi:\ran(f)\to\ran(g)$ such that $\psi f=g\phi$, and \emph{weakly Borel equivalent} if there exist Borel measurable such $\phi$ and $\psi$. Clearly (Borel) equivalent functions are weakly (Borel) equivalent, and the converse can fail. Now the Borel functions $f$ and $g$ are weakly equivalent if and only if $\fs(\Eq{f})=\fs(\Eq{g})$, and weakly Borel equivalent if and only if $\Eq{f}\cong_B\Eq{g}$. We will return to these considerations at the end of Section \ref{sec:pairs2}.

\section{Combinatorics of pairs of countable Borel equivalence relations}\label{sec:pairs0}

In this section we introduce for pairs $E\subseteq F$ of countable Borel equivalence relations an analogue of ``shape" as it was defined in Section \ref{sec:singletons}. Here the notion of shape will be more complicated, since we will need to consider both local shapes describing the distribution of $E$-classes inside a particular $F$-class, and global shapes describing the distribution of local shapes across all the $F$-classes.

\begin{definition} Let $E\subseteq F$ be countable Borel equivalence relations on the standard Borel space $X$, and let $C\subseteq X$ be an $F$-class. Define the \emph{local fine (relative) shape} of $C$ to be the sequence
\[
\lfs^{(E,F)}(C) \ := \ \langle \n_m(E\res C)\st 1\leq m\leq\omega\rangle
\]
and the \emph{local coarse (relative) shape} of $C$ to be the sequence
\[
\lcs^{(E,F)}(C) \ := \ \langle \n_{\geq m}(E\res C)\st 1\leq m\leq\omega\rangle.
\]
\end{definition}

If the equivalence relations $E\subseteq F$ are clear from context, we might omit reference to them in the notation and write simply $\lfs(C)$ or $\lcs(C)$. Note that the local shape (fine or coarse) of an equivalence class is a function from $\C^+\setminus\{\cc\}$ to $\C\setminus\{\cc\}$ that is not constantly zero. We let $\Sf$ denote the set of all possible local fine shapes of equivalence classes, viewed as a Polish space in the obvious way as a homeomorph of Baire space. Equivalently, 
\[
\Sf \ := \ \{\fs(E)\res (\C^+\setminus\{\cc\}) \st \mbox{$E$ is an equivalence relation on a countable set}\}.
\]
Let also $\Sc$ denote the set of all possible local coarse shapes. Since
\[
\Sc\:\subseteq\:\{\alpha\in\Sf \st \mbox{$\alpha(m)\geq\alpha(n)$ whenever $m\leq n$}\},
\]
$\Sc$ is a countable subset of $\Sf$. We view $\Sc$ as a Polish space with the discrete topology. Below we will sometimes use self-explanatory notation such as $\langle n^k,\bar{m},l\rangle$ to denote the function $\alpha$ from $\C^+\setminus\{\cc\}$ to $\C\setminus\{\cc\}$ such that
\[
\alpha(i)=\left\{\begin{array}{ll} n & \mbox{if $1\leq i\leq k$;} \\ m & \mbox{if $k<i<\omega$;} \\ l & \mbox{if $i=\omega$.} \end{array}\right.
\]
Note that $\Sc\ne\{\alpha\in\Sf \st \alpha(m)\geq\alpha(n)\mbox{ whenever $m\leq n$}\}$ since, for instance, $\langle\bar{1},0\rangle\not\in\Sc$.

\begin{lemma} Let $E\subseteq F$ be countable Borel equivalence relations on the standard Borel space $X$. The functions
\[
\begin{array}{lll}
X & \rightarrow & \Sc \\
x & \mapsto & \lcs([x]_F)
\end{array}\quad\mbox{and}\quad
\begin{array}{lll}
X & \rightarrow & \Sf \\
x & \mapsto & \lfs([x]_F)
\end{array}
\]
are Borel. \label{lem:PBor} \end{lemma}

\begin{proof} Fix $E\subseteq F$ as in the statement of the lemma. For each $1\leq m,n\leq\omega$, let
\[
\begin{array}{lll}
P_{n,m}      & = & \{x\in X \st [x]_F\mbox{ has exactly $n$ $E$-classes containing exactly $m$ elements}\}, \\
P_{n,\geq m} & = & \{x\in X \st [x]_F\mbox{ has exactly $n$ $E$-classes containing at least $m$ elements}\}.
\end{array}
\]
We claim that for each $m,n$, $P_{n,m}$ and $P_{n,\geq m}$ are Borel. To see this, fix $m,n$. Let $\pi:X\to X/E$ be the quotient map, and let $F=E^X_{\Gamma}$ where $\Gamma=\{\gamma_i\st i\in\omega\}$. Let $Y_m=X^{(E)}_m/\big(E\res X^{(E)}_m\big)$, and let $D_m$ be the equivalence relation $D_m=\big(F\res X^{(E)}_m)/(E\res X^{(E)}_m\big)$ on $Y_m$. Then
\[
x\in P_{n,m} \quad \Leftrightarrow \quad \exists i\,\pi(\gamma_i\cdot x) \ \in \ (Y_m)^{(D_m)}_n \ = \ \{ C\in Y_{m}\st |[C]_{D_{m}}|=n\}.
\]
Similarly, letting $Y_{\geq m}=X^{(E)}_{\geq m}\,/\big(E\res X^{(E)}_{\geq m}\big)$ and $D_{\geq m}=\big(F\res X^{(E)}_{\geq m}\big)/\big(E\res X^{(E)}_{\geq m}\big)$, we have
\[
x\in P_{n,\geq m} \quad \Leftrightarrow \quad \exists i\,\pi(\gamma_i\cdot x) \ \in \ (Y_{\geq m})^{(D_{\geq m})}_n \ = \ \{ C\in Y_{\geq m}\st |[C]_{D_{\geq m}}|=n\}.
\]
By \ref{fact:ctblPartition}, this shows that $P_{n,m}$ and $P_{n,\geq m}$ are Borel. Since $\Sc$ is countable discrete and for each $\alpha\in\Sc$ we have
\[
\lcs([x]_F)=\alpha\quad\Leftrightarrow\quad x\in\bigcap_{1\leq m\leq\omega}P_{\alpha(m),\geq m},
\]
it follows that $x\mapsto\lcs([x]_F)$ is Borel.

To see that $x\mapsto\lfs([x]_F)$ is Borel, for each function $s$ mapping a finite subset  of $\C^+\setminus\{\cc\}$ into $\C\setminus\{\cc\}$, let
\[
U_s = \{\alpha\in\Sf\st\alpha\res\dom(s)=s\}.
\]
The sets $U_s$ form a base for the topology on $\Sf$, so we must show that for each such $s$, the set $\{x\in X \st \lfs([x]_F)\in U_s\}$ is Borel in $X$. But clearly
\[
\lfs([x]_F)\in U_s\quad\Leftrightarrow\quad x\in\bigcap_{m\,\in\,\mbox{{\scriptsize dom}}(s)}P_{s(m),m}. \qedhere
\]
\end{proof}

\begin{remark} It follows from Lemma \ref{lem:PBor} that if $E\subseteq F$ are smooth countable equivalence relations on $X$, then the functions $[x]_F\mapsto\lcs([x]_F)$ and $[x]_F\mapsto\lfs([x]_F)$ from $X/F$ to $\Sc$ and $\Sf$, respectively, are Borel. \end{remark}

Now we are ready to define the ``global" notions of shape that will serve as combinatorial invariants for simultaneous Borel reducibility, embeddability, and isomorphism of pairs of countable Borel equivalence relations.

\begin{definition}\label{df:crs} Let $E\subseteq F$ be countable Borel equivalence relations on the standard Borel space $X$. For each $1\leq m\leq\omega$, let $\n_{\geq m}(E,F)$ be the number of $F$-classes that contain at least $m$ $E$-classes, and define the \emph{coarse (relative) shape} of $(E,F)$ to be the sequence
\[
\crs(E,F) \ := \ \langle \n_{\geq m}(E,F) \st 1\leq m\leq\omega\rangle.
\]
\end{definition}

The fact that each $\n_{\geq m}(E,F)$ belongs to $\C$ follows from Lemma \ref{lem:PBor}. Notice that when $E$ is smooth, $\crs(E,F)=\cs(F/E)$.

Recall that the set $\Sc$ of all local coarse shapes is a countable subset of $\C^{\C^+\setminus\{\cc\}}$ that is partially ordered by the relation $\alpha\leq\beta\;\Leftrightarrow\;(\forall m)\,\alpha(m)\leq\beta(m)$. An \emph{upper set} in $\Sc$ is any subset $W\subseteq\Sc$ such that for all $\alpha,\beta\in\Sc$, if $\alpha\in W$ and $\alpha\leq\beta$ then $\beta\in W$. Let $\uparrow\!\Sc$ denote the collection of all upper sets in $\Sc$, and for $\alpha\in\Sc$ write $\uparrow\!\alpha$ for the set of all $\beta\in\Sc$ such that $\alpha\leq\beta$.

\begin{definition}\label{df:gcs} Let $E\subseteq F$ be countable Borel equivalence relations on the standard Borel space $X$. For each $W\in\,\uparrow\!\Sc$, let $\n_W(E,F)$ be the number of $F$-classes $C$ for which $\lcs(C)\in W$. Define the \emph{global coarse (relative) shape} of $(E,F)$ to be the function
\[
\begin{array}{llllll}
\gcs(E,F)  & : & \uparrow\!\Sc & \rightarrow & \C            \\
           &   & W             & \mapsto     & \n_W(E,F).
\end{array}
\]
\end{definition}

Notice that $\n_W(E,F)$ belongs to $\C$ because the function $x\mapsto\lcs([x]_F)$ is Borel and each $W\subseteq\Sc$ is Borel. In Section \ref{sec:pairs1} we will see that $\uparrow\!\Sc$ is countable, which will greatly simplify our analysis of global coarse shape.

\begin{definition}\label{df:gfs} Let $E\subseteq F$ be countable Borel equivalence relations on the standard Borel space $X$. For each $\alpha\in\Sf$, let $\n_\alpha(E,F)$ be the number of $F$-classes $C$ for which $\alpha=\lfs(C)$. Define the \emph{global fine (relative) shape} of $(E,F)$ to be the function
\[
\begin{array}{llllll}
\gfs(E,F)  & : & \Sf     & \rightarrow & \C            \\
           &   & \alpha  & \mapsto     & \n_\alpha(E,F).
\end{array}
\]
\end{definition}

Notice that $\n_\alpha(E,F)$ belongs to $\C$ because $x\mapsto\lfs([x]_F)$ is Borel. We now show that the notions of shape introduced in Definitions \ref{df:crs}, \ref{df:gcs}, and \ref{df:gfs} provide necessary conditions for the existence of simultaneous Borel reductions, embeddings, and isomorphisms, respectively, between pairs $E\subseteq F$ of countable Borel equivalence relations.

\begin{proposition}\label{prop:shape2} Let $E\subseteq F$ and $E'\subseteq F'$ be countable Borel equivalence relations on the standard Borel spaces $X$ and $Y$, respectively.
\begin{enumerate}
 \item[(i)] If $(E,F)\leq_B (E',F')$, then $\crs(E,F)\leq \crs(E',F')$.
 \item[(ii)] If $f:X\to Y$ is a simultaneous Borel embedding from $(E,F)$ to $(E',F')$, then for every $F$-class $C$ we have $\lcs^{(E,F)}(C)\leq\lcs^{(E',F')}([f(C)]_{F'})$.
 \item[(iii)] If $(E,F)\sqsubseteq_B (E',F')$, then $\gcs(E,F)\leq \gcs(E',F')$.
 \item[(iv)] If $(E,F)\cong_B(E',F')$, then $\gfs(E,F)=\gfs(E',F')$.
\end{enumerate}
\end{proposition}

\begin{proof} (i) Suppose $f:X\to Y$ is a simultaneous Borel reduction from $(E,F)$ to $(E',F')$ with $\tilde{f}:X/F\to Y/F'$ the induced embedding of classes, and let $1\leq m\leq \omega$ be arbitrary. For each $F$-class $C\subseteq X$ containing at least $m$ $E$-classes, $\tilde{f}(C)$ must contain at least $m$ $E'$-classes since $f$ reduces $E$ to $E'$. Hence $\tilde{f}$ injectively maps the collection of $F$-classes containing at least $m$ $E$-classes into the collection of $F'$-classes containing at least $m$ $E'$-classes. Thus $\n_{\geq m}(E,F)\leq\n_{\geq m}(E',F')$, and since $m$ was arbitrary we conclude that $\crs(E,F)\leq \crs(E',F')$.

(ii) Fixing $C$, the conclusion follows immediately from (i) if we view $f\res C$ as a simultaneous Borel reduction from $(\Delta(C),\,E\res C)$ to $(\Delta([f(C)]_{F'}),\,E'\res [f(C)]_{F'})$.

(iii) Let $f:X\to Y$ be a simultaneous Borel embedding from $(E,F)$ to $(E',F')$ with induced embedding $\tilde{f}:X/F\to Y/F'$. Let $W$ be an arbitrary upper set in $\Sc$, and suppose that $C\subseteq X$ is an $F$-class such that $\lcs(C)\in W$. Since $\lcs^{(E,F)}(C)\leq\lcs^{(E',F')}(\tilde{f}(C))$ by (ii), we have $\lcs^{(E',F')}(\tilde{f}(C))\in W$. Since $\tilde{f}$ is injective, this shows that $\n_W(E,F)\leq\n_W(E',F')$. As $W$ was arbitrary, the result follows.

(iv) Suppose $f$ is a simultaneous Borel isomorphism of $(E,F)$ with $(E',F')$. Then for each $F$-class $C$ in $X$, we have $\lfs(C)=\lfs(f(C))$. Thus for each $\alpha\in\Sf$, $f$ induces a bijection between those $F$-classes $C\subseteq X$ for which $\lfs(C)=\alpha$ and those $F'$-classes $C'\subseteq Y$ for which $\lfs(C')=\alpha$. In particular, for each $\alpha\in\Sf$ we have $\n_\alpha(E,F)=\n_\alpha(E',F')$, and hence $\gfs(E,F)=\gfs(E',F')$, as desired. \end{proof}

Note that if $(E,F)\sqsubseteq_B(E',F')$, then of course also $E\sqsubseteq_BE'$, $F\sqsubseteq_BF'$, and $(E,F)\leq_B(E',F')$, so by Propositions \ref{prop:shape} and \ref{prop:shape2} we have
\begin{equation*}
\cs(E)\leq \cs(E'), \quad \cs(F)\leq \cs(F'),\quad\mbox{and}\quad \crs(E,F)\leq \crs(E',F').
\end{equation*}
In Section \ref{sec:pairs1} we will show that if $E\subseteq F$ and $E'\subseteq F'$ are smooth countable equivalence relations, then $\gcs(E,F)\leq\gcs(E',F')$ implies $(E,F)\sqsubseteq_B (E',F')$. Now we show that the three inequalities above follow easily from the assumption that $\gcs(E,F)\leq\gcs(E',F')$ even if these equivalence relations are not smooth.

\begin{proposition}\label{prop:gscTechnical} Let $E\subseteq F$ and $E'\subseteq F'$ be countable Borel equivalence relations on the standard Borel spaces $X$ and $Y$, respectively. If $\gcs(E,F)\leq \gcs(E',F')$ then $\cs(E)\leq \cs(E')$, $\cs(F)\leq \cs(F')$, and $\crs(E,F)\leq \crs(E',F')$. \end{proposition}

\begin{proof} To see that $\crs(E,F)\leq \crs(E',F')$, notice that for each $1\leq m\leq\omega$ we have
\[
\begin{array}{lllllll}
\n_{\geq m}(E,F) & = & \n_{\uparrow\langle m,\bar{0},0\rangle}(E,F) & \leq & \n_{\uparrow\langle m,\bar{0},0\rangle}(E',F')
& = & \n_{\geq m}(E',F').
\end{array}
\]

To see that $\cs(F)\leq \cs(F')$, for each $1\leq m\leq\omega$ let $W_m$ be the set of all $\alpha\in\Sc$ such that any class having local coarse shape $\alpha$ contains at least $m$ elements. Then each $W_m$ is an upper set, so for each $1\leq m\leq\omega$ we have
\[
\begin{array}{lllllll}
\n_{\geq m}(F) & = & \n_{W_m}(E,F) & \leq & \n_{W_m}(E',F') & = & \n_{\geq m}(F').
\end{array}
\]

Finally, to see that $\cs(E)\leq \cs(E')$, fix $1\leq m\leq\omega$, let $1\leq k<\omega$, and let $W_k=\,\uparrow\!\langle k^m,\bar{0},0\rangle$ if $m<\omega$ and $W_k=\,\uparrow\!\langle \bar{k},k\rangle$ if $m=\omega$. Notice that for any $F$-class $C$, $\lcs(C)\in W_k$ iff $C$ has at least $k$ $E$-classes containing at least $m$ elements. Therefore $\n_{W_k}(E,F)$ is the number of $F$-classes that have at least $k$ $E$-classes containing at least $m$ elements. It follows that
\[
\begin{array}{lllllll}
\n_{\geq m}(E) & = & \displaystyle{\sum_{1\leq k<\omega}\n_{W_k}(E,F)} & \leq &
\displaystyle{\sum_{1\leq k<\omega}\n_{W_k}(E',F')} & = & \n_{\geq m}(E').
\end{array}\qedhere
\]
\end{proof}

Before proceeding to consider the converses of (i), (iii), and (iv) of Proposition \ref{prop:shape2} in the next two sections, we explicitly state an easy technical lemma that will be useful later.

\begin{lemma} Let $E\subseteq F$ and $E'\subseteq F'$ be smooth countable equivalence relations on the standard Borel spaces $X$ and $Y$, respectively, and suppose $\phi:X\to Y$ is a Borel reduction from $F$ to $F'$.
\begin{enumerate}
 \item[(i)] If $\lcs([x]_F)\leq\lcs([f(x)]_{F'})$ for all $x\in X$, then there is a simultaneous Borel embedding $\psi:X\to Y$ from $(E,F)$ to $(E',F')$ such that for all $x\in X$, $\phi(x)\mathrel{F'}\psi(x)$;
 \item[(ii)] If $\lfs([x]_F)=\lfs([f(x)]_{F'})$ for all $x\in X$, then there is a simultaneous Borel isomorphism $\psi:X\to Y\res [\ran(\phi)]_{F'}$ from $(E,F)$ to $(E',F')\res [\ran(\phi)]_{F'}$ such that for all $x\in X$, $\phi(x)\mathrel{F'}\psi(x)$.
\end{enumerate} \label{lem:adjust}
\end{lemma}

\begin{proof} We argue as in the proof of Case 1 of Proposition \ref{prop:smoothemb}. Using \ref{fact:ctblSmooth}, well-order in a uniform Borel manner each $F$-class in order type finite or $\omega$, and do likewise with each $F'$-class. Let $C=\{x_n\st n<|C|\}$ be a particular $F$-class, indexed according to our ordering, and likewise write $D=[\phi(C)]_{F'}=\{y_n\st n<|D|\}$. Working on $C$, inductively define $\psi(x_n)$ as follows. If there is no $m<n$ such that $x_m\mathrel{E}x_n$, let $\m\geq |[x_n]_E|$ be least for which there exists an $E'$-class in $D$ of size $\m$ disjoint from $\{\psi(x_i)\st i<n\}$, and define $\psi(x_n)$ to be the least-indexed element in $D$ that belongs to such an $E'$-class; otherwise, if $m<n$ is such that $x_m\mathrel{E}x_n$, let $f(x_n)$ be the least-indexed element in $[\psi(x_m)]_{E'}\setminus \{\psi(x_i)\st i<n\}$. The hypotheses for (i) and (ii) guarantee that such choices always exist. The function $\psi$ thus defined is Borel since the well-orderings of the $F$-classes and the $F'$-classes were Borel. Finally, it is easy to check that $\psi$ satisfies the conclusions of (i) and (ii) provides that the corresponding hypotheses hold. \end{proof}

\section{Classifying smooth countable pairs up to $\sim_B$ and $\approx_B$}\label{sec:pairs1}

In this section we establish the converses of parts (i) and (iii) of Proposition \ref{prop:shape2} for pairs $E\subseteq F$ of smooth countable equivalence relations. The fact that coarse relative shape is a complete invariant for $\sim_B$ will follow easily from Proposition \ref{prop:smoothemb}, and most of our effort will be spent on $\approx_B$, which involves some interesting combinatorics.

\begin{theorem}\label{thm:SmoothPairsRedSuf} Let $E\subseteq F$ and $E'\subseteq F'$ be smooth countable equivalence relations on the standard Borel spaces $X$ and $Y$, respectively. Then $(E,F)\leq_B(E',F')$ if and only if $\crs(E,F)\leq \crs(E',F')$. In particular, $(E,F)\sim_B (E',F')$ if and only if $\crs(E,F)=\crs(E',F')$. \end{theorem}

\begin{proof} The forward direction holds for any pairs of countable Borel equivalence relations by Proposition \ref{prop:shape2}. For the converse, consider the smooth equivalence relations $F/E$ and $F'/E'$ on the standard Borel spaces $X/E$ and $Y/E'$, respectively. Since
\[
\crs(E,F)=\cs(F/E)\quad\mbox{and}\quad \crs(E',F')=\cs(F'/E'),
\]
by Proposition \ref{prop:smoothemb} there is a Borel embedding $\tilde{f}:X/E\rightarrow Y/E'$ from $F/E$ into $F'/E'$. Letting $T\subseteq Y$ be a Borel transversal for $E'$, define $f:X\to Y$ by
\[
f(x) \ = \ \mbox{the unique element of $T$ that belongs to $\tilde{f}([x]_E)$}.
\]
Then $f$ is a simultaneous Borel reduction from $(E,F)$ to $(E',F')$. \end{proof}

The remainder of this section is devoted to showing that global coarse shape is a complete invariant for the simultaneous biembeddability relation on pairs of smooth countable equivalence relations. We begin by proving that $(\Sc,\leq)$ is a well partial order, i.e., that $(\Sc,\leq)$ is well-founded and has no infinite antichains, which will imply that $\uparrow\!\Sc$ is countable. Recall that $\Sc$ is the countable set of all function $\alpha:\{1,\ldots,\omega\}\to\{0,\ldots,\omega\}$ such that $\alpha\ne\bar{0}$, $\alpha(m)\geq\alpha(n)$ whenever $m\leq n$, and $\alpha(\omega)=\lim\alpha(n)$, ordered by $\alpha\leq\beta\Leftrightarrow (\forall m)\,\alpha(m)\leq\beta(m)$.

\begin{lemma} $(\Sc,\leq)$ is well-founded. \label{lem:WF} \end{lemma}

\begin{proof} Suppose $\langle\alpha_k\st k\in\omega\rangle$ is a strictly decreasing sequence in $\Sc$. For each $k\in\omega$, let $m(k)$ be the least $m$ such that $\alpha_k(m)>\alpha_{k+1}(m)$. Passing to a subsequence of $\langle m(k)\st k\in\omega\rangle$ if necessary, we may assume that $\langle m(k)\st k\in\omega\rangle$ is (weakly) increasing. But then $\langle\alpha_{k+1}(m(k))\st k\in\omega\rangle$ is an infinite strictly decreasing sequence of natural numbers. \end{proof}

\begin{lemma} $(\Sc,\leq)$ has no infinite antichains. \label{lem:antichain} \end{lemma}

\begin{proof} Suppose for contradiction that $\{\alpha_k\st k\in\omega\}\subseteq\Sc$ is pairwise $\leq$-incomparable. At most one $\alpha_k$ can be of the form $\langle\bar{\omega},n\rangle$, so without loss of generality none of them have this form. This means that for every $k$ there is $n(k)\in\omega$ such that for all but finitely many $i\in\omega$, $\alpha_k(i)=\alpha_k(\omega)=n(k)$.

Now, suppose first that the mapping $k\mapsto n(k)$ has unbounded range. For arbitrary $k$, if $\alpha_k(0)$ were finite then fixing $j$ such that $\alpha_k(0)\leq n(j)$ we would have $\alpha_k(i)\leq\alpha_j(i)$ for all $i$, so $\alpha_k(0)$ must be $\omega$. But this is true for all $k$, so by continuing this argument inductively we see that $\alpha_k(i)=\omega$ for all $k$ and $i$, a contradiction.

Hence we may assume that there is $N\in\omega$ such that $n(k)=N$ for infinitely many $k$. Fix such $N$, and by passing to a subsequence suppose $n(k)=N$ for all $k$. For each $k$, let $m(k)\in\omega$ be least such that $\alpha_k(i)=N$ for all $i\geq m(k)$. Now we will inductively define a decreasing sequence of infinite subsets $A_j\subseteq\omega$ along with integers $k_j$ and $\ell_j$ as follows. For the induction basis, let $A_0=\omega$, let $k_0\in A_0$ be such that $m(k_0)$ is least in $\{m(k)\st k\in A_0\}$, and let $\ell_0<m(k_0)$ be such that $(\exists^\infty k\in A_0)\,\alpha_k(\ell_0)<\alpha_{k_0}(\ell_0)$. There must exist such $\ell_0$ since for every $k\ne k_0$ there is some $\ell$ for which $\alpha_k(\ell)<\alpha_{k_0}(\ell)$, while for all $i\geq m(k_0)$, $\alpha_{k_0}(i)\leq\alpha_k(i)$. Then supposing $A_j$, $k_j$, and $\ell_j$ have been defined, let
\[
\begin{array}{l}
A_{j+1} \ = \ \{k\in A_j\st \alpha_k(\ell_j)<\alpha_{k_j}(\ell_j)\}; \\
k_{j+1} \ \in \ A_{j+1} \mbox{ is such that $m(k_{j+1})$ is least in $\{m(k)\st k\in A_{j+1}\}$}; \\
\ell_{j+1} \ < \ m(k_{j+1}) \mbox{ is such that $(\exists^\infty k\in A_{j+1})\,\alpha_k(\ell_{j+1})<\alpha_{k_{j+1}}(\ell_{j+1})$}.
\end{array}
\]
Now let $\langle \ell_{j(i)}\st i\in\omega\rangle$ be a subsequence of $\langle \ell_j\st j\in\omega\rangle$ that is (weakly) increasing. Then for all $i\in\omega$ we have
\[
\alpha_{k_{j(i)}}(\ell_{j(i)}) \ > \ \alpha_{k_{j(i+1)}}(\ell_{j(i)}) \ \geq \ \alpha_{k_{j(i+1)}}(\ell_{j(i+1)}),
\]
so
\[
\langle\alpha_{k_{j(i)}}(\ell_{j(i)})\st i\in\omega\rangle
\]
is an infinite strictly decreasing sequence of natural numbers, a contradiction. \end{proof}

Together Lemmas \ref{lem:WF} and \ref{lem:antichain} imply that any upper set in $\Sc$ has the form
\[
\{\beta\in\Sc\st \alpha_0\leq\beta\vee\cdots\vee\alpha_k\leq\beta\}
\]
for some $\alpha_0,\ldots,\alpha_k\in\Sc$. In particular, since $\Sc$ is countable it follows that $\uparrow\!\Sc$ is also countable. Writing $\bar{\alpha}=(\alpha_0,\ldots,\alpha_k)$, we will use the notation
\[
\uparrow\!(\bar{\alpha})=\{\beta\in\Sc\st \alpha_0\leq\beta\vee\cdots\vee\alpha_k\leq\beta\}
\]
for the upper set determined by $\bar{\alpha}$. We view $\uparrow\!\Sc$ as a Polish space with the discrete topology, and we think of $\gcs(E,F)$ as the ``sequence" of values $\n_{\uparrow(\bar{\alpha})}(E,F)$ as $\uparrow\!(\bar{\alpha})$ varies over $\uparrow\!\Sc$.

\begin{theorem} Let $E\subseteq F$ and $E'\subseteq F'$ be smooth countable equivalence relations on the standard Borel spaces $X$ and $Y$, respectively. Then $(E,F)\sqsubseteq_B(E',F')$ if and only if $\gcs(E,F)\leq\gcs(E',F')$. In particular, $(E,F)\approx_B(E',F')$ if and only if $\gcs(E,F)=\gcs(E',F')$. \label{thm:SmoothPairsEmbSuf} \end{theorem}

Our proof of Theorem \ref{thm:SmoothPairsEmbSuf} will make use of Hall's marriage theorem, which we state here in the following form. Let $\mathcal{A}=\{A_i\st i\in I\}$ be an indexed family of finite sets. A \emph{system of distinct representatives} for $\mathcal{A}$ is an injective function $s:I\to\bigcup_{i\in I}A_i$ such that for all $i\in I$, $s(i)\in A_i$. The \emph{marriage condition} for $\mathcal{A}$ is the statement that for all subsets $J\subseteq I$, $|J|\leq|\bigcup_{i\in J}A_i|$. Then Hall's theorem states that $\mathcal{A}$ admits a system of distinct representatives if and only if $\mathcal{A}$ satisfies the marriage condition.

\begin{proof}[Proof of Theorem \ref{thm:SmoothPairsEmbSuf}] The forward direction follows from Proposition \ref{prop:shape2}(iii). As a preliminary step in proving the converse, we will split the spaces $X$ and $Y$ into two invariant Borel sets, and work on each separately. Thus we let
\[
S_0 \ = \ \{\alpha\in\Sc\st \n_{\uparrow\alpha}(E',F')<\omega\}
\]
and define
\[
X_0 \ = \ \{x\in X\st \lcs([x]_F)\in S_0\} \quad \mbox{and} \quad Y_0 \ = \ \{y\in Y\st \lcs([y]_{F'})\in S_0\}.
\]
We also let
\[
S_1=\Sc\setminus S_0,\quad X_1=X\setminus X_0,\quad\mbox{and}\quad Y_1=Y\setminus Y_0.
\]
Notice that $S_0$ is an upper subset of $\Sc$, and hence we may fix $\bar{\alpha}=(\alpha_0,\ldots,\alpha_k)$ such that $S_0=\;\uparrow\!\!(\bar{\alpha})$. Then each $\alpha_i$ belongs to $S_0$, and $S_0=\bigcup_{i\leq k}\!\uparrow\!\alpha_i$. It follows that there are at most finitely many $F'$-classes with local coarse shape in $S_0$, and hence by our assumption that $\gcs(E,F)\leq\gcs(E',F')$, there are at most finitely many $F$-classes in $X_0$. In particular, $X_0$ is a countable $F$-invariant subset of $X$. We will work separately on $X_0$ and on $X_1$, obtaining simultaneous Borel embeddings $f_0:X_0\to Y_0$ from $(E,F)\res X_0$ to $(E',F')\res Y_0$ and $f_1:X_1\to Y_1$ from $(E,F)\res X_1$ to $(E',F')\res Y_1$. For $i\in\{0,1\}$, let us write
\[
E_i=E\res X_i,\quad F_i=F\res X_i,\quad E_i'=E'\res Y_i,\quad F_i'=F'\res Y_i.
\]
The key fact that will allow us to work separately on the two pieces is the observation that
\[
\gcs(E_0,F_0)\quad\leq\quad\gcs(E_0',F_0')
\]
and
\begin{equation}
\gcs(E_1,F_1)\quad\leq\quad\gcs(E_1',F_1').
\tag{$\ast$}
\end{equation}
To see that the first inequality is true, note that for any upper set $W$ in $\Sc$, the set $W\cap S_0$ is again an upper set, and
\[
\begin{array}{llllllll}
     & \n_W(E_0,F_0)         & = & \n_{W\cap S_0}(E_0,F_0)   & = & \n_{W\cap S_0}(E,F) & & \\
& & \leq & \n_{W\cap S_0}(E',F') & = & \n_{W\cap S_0}(E_0',F_0') & = & \n_W(E_0',F_0').
\end{array}
\]
For the second inequality, if $W$ is any upper set contained in $S_0$ then
\[
\n_W(E_1,F_1) \ = \ \n_W(E_1',F_1') \ = \ 0,
\]
and if $W$ is any upper set in $\Sc$ such that $W\cap S_1\ne\emptyset$, then $\n_W(E',F')$ is infinite, which implies that
\[
\begin{array}{lllllll}
\n_W(E_1,F_1) & \leq & \n_W(E,F) & \leq & \n_W(E',F') & = & \n_W(E_1',F_1').
\end{array}
\]

Now that we have broken the problem into two pieces, first we consider defining a simultaneous Borel embedding $f_0:X_0\to Y_0$ from $(E_0,F_0)$ to $(E_0',F_0')$. We have that $X_0/F_0$ and $Y_0/F_0'$ are finite, so $X_0$ and $Y_0$ are countable. Hence any map defined on $X_0$ will be Borel, so the problem is purely one of combinatorics. In fact, the problem is exactly the one addressed by Hall's marriage theorem. Using the notation described above, we have $I=X_0/F_0$ and for each $i=[x]_{F_0}\in I$,
\[
A_i=\left\{\;[y]_{F_0'}\in Y_0/F_0'\::\: \lcs^{(E_0,F_0)}([x]_{F_0})\,\leq\,\lcs^{(E_0',F_0')}([y]_{F_0'})\;\right\}.
\]
Now let $J\subseteq I$ be arbitrary, and let $J'$ be the upward closure of $J$ in $X_0/F_0$, i.e.,
\[
J' \ = \ \{\,[x]_F\in X_0/F_0\st (\exists C\in J)\,\lcs(C)\leq\lcs([x]_F)\,\}\quad\subseteq\quad X_0/F_0.
\]
Using now the fact that $\gcs(E_0,F_0)\leq\gcs(E_0',F_0')$, we have for each $J\subseteq I$ that
\[
|J|\quad\leq\quad |J'|\quad\leq\quad \left|\bigcup_{i\in J}A_i\right|.
\]
But this is exactly the marriage condition for $\mathcal{A}=\{A_i\st i\in I\}$, so by Hall's theorem there exists a system of distinct representatives $s:I\to\bigcup_{i\in I}A_i$ for $\mathcal{A}$. Then $s$ is an injection from $X_0/F_0$ to $Y_0/F_0'$ such that for all $C\in X_0/F_0$, $\lcs(C)\leq\lcs(s(C))$. We conclude that $(E_0,F_0)\sqsubseteq_B(E_0',F_0')$ by Lemma \ref{lem:adjust}.

Next we turn to the construction of a simultaneous Borel embedding $f_1:X_1\to Y_1$ from $(E_1,F_1)$ to $(E_1',F_1')$. Recycling notation, for $\alpha\in S_1\subseteq\Sc$ let us temporarily write $X_\alpha$ for the set of $x\in X_1$ such that $\lcs([x]_F)=\alpha$, and let $S_1^*=\{\alpha\in S_1\st X_\alpha\ne\emptyset\}$. $S_1^*$ is countable, so let $\langle\alpha_k\st k<N\rangle$ enumerate $S_1^*$, where $N=|S_1^*|\in\omega\cup\{\omega\}$. We will define, for each $k<N$, a simultaneous Borel embedding $g_k:X_{\alpha_k}\to Y_1$ from $(E_1,F_1)\res X_{\alpha_k}$ to $(E_1',F_1')$. For notational convenience we put
\[
X(k) \, := \, \bigcup_{i\geq k} X_{\alpha_i} \quad \mbox{and} \quad Z(k) \, := \, Y_1\setminus\bigcup_{i<k}[\ran(g_i)]_{F'}.
\]
We will define the functions $g_k$ by induction on $k$ so that for each $k<N$, $\ran(g_k)\subseteq Z(k)$ and
\begin{equation}
\gcs\big(\,(E_1',F_1')\res Z(k+1)\,\big) \quad = \quad \gcs\big(E_1',F_1'\big).
\tag{$\ast\ast$}
\end{equation}
Supposing we have done this, $\bigcup_k g_k$ will then be a simultaneous Borel embedding from $(E_1,F_1)$ to $(E_1',F_1')$, as desired. Thus fix $k<N$ and suppose we have defined $g_i$ for each $i<k$ satisfying the stated conditions. In particular, note by ($\ast$) and ($\ast\ast$) that
\[
\gcs\big(\,(E_1,F_1)\res X(k)\,\big) \ \leq \ \gcs(E_1,F_1) \ \leq \ \gcs(E_1',F_1') \ = \ \gcs\big(\,(E_1',F_1')\res Z(k)\,\big).
\]
This implies that
\[
\n_{\uparrow\alpha_k}\big(\,(E_1',F_1')\res Z(k)\,\big) \ = \ \n_{\uparrow\alpha_k}(E_1',F_1'),
\]
and therefore since $\alpha_k\in S_1$, $\n_{\uparrow\alpha_k}\big((E',F')\res Z(k)\big)$ is infinite.

Now we consider two cases. First suppose $\n_{\uparrow\alpha_k}\big((E',F')\res Z(k)\big)=\cc$. Then since $\uparrow\!\alpha_k$ is countable, there must be $\beta\in\,\uparrow\!\alpha_k$ such that the number of $F'$-classes in $Z(k)$ having local coarse shape $\beta$ is $\cc$. Fix such $\beta$, and let $U$ be an $F'$-invariant Borel subset of $Z(k)$ such that every $F'$-class in $U$ has local coarse shape $\beta$ and both $U$ and $Z(k)\setminus U$ contain uncountably many $F'$-classes having local coarse shape $\beta$. Since $|U/F'|=\cc$, there is a Borel reduction $\phi:X_{\alpha_k}\to U$ from $F_1\res X_{\alpha_k}$ to $F_1'\res U$, and since $\alpha\leq\beta$ this reduction can be adjusted using Lemma \ref{lem:adjust} to obtain a Borel embedding $g_k$ from $(E_1,F_1)\res X_{\alpha_k}$ into $(E_1',F_1')\res U$, as desired. Finally, since every $F'$-class in $U\subseteq Z(k)\setminus Z(k+1)$ has local coarse shape $\beta$ but $Z(k)\setminus U\subseteq Z(k+1)$ contains uncountably many $F'$-classes with local coarse shape $\beta$, ($\ast\ast$) continues to hold after the construction of $g_k$.

For the second case, suppose $\n_{\uparrow\alpha_k}\big((E',F')\res Z(k)\big)=\aleph_0$. If there is any particular $\beta\in\,\uparrow\!\alpha_k$ for which there exist infinitely many $F'$-classes in $Z(k)$ having local coarse shape $\beta$, then as above we can let $U$ be an $F'$-invariant subset of $Z(k)$ such that every $F'$-class in $U$ has local coarse shape $\beta$ and both $U$ and $Z(k)\setminus U$ contain infinitely many $F'$-classes having local coarse shape $\beta$, and define $g_k$ using \ref{lem:adjust} so that it takes values in $U$. In this case $(\ast\ast)$ will continue to hold for the same reason as above, namely we will have $\gcs((E',F')\res Z(k))=\gcs((E',F')\res Z(k)\setminus U)$. So we may assume that for every $\beta\in\,\uparrow\!\alpha_k$ there are at most finitely many $F'$-classes in $Z(k)$ having local coarse shape $\beta$. Let $V$ be the set of all $\beta\in\,\uparrow\!\alpha_k$ for which there exists at least one $F'$-class in $Z(k)$ having local coarse shape $\beta$, so that $V\subseteq\Sc$ is infinite. Since $(V,\leq)$ is a well partial order, there exist $\beta_0,\ldots,\beta_m\in V$ such that $V\setminus\{\alpha_k\}=\,\uparrow\!(\beta_0,\ldots,\beta_m)$. By the pigeon-hole principle, there must be $i\in\{0,\ldots,m\}$ for which there are infinitely many $F'$-classes in $Z(k)$ having local coarse shape in $\uparrow\!\beta_i$. Continuing inductively we obtain an increasing sequence $\alpha_k=\gamma_0<\gamma_1<\gamma_2<\gamma_3<\cdots$ of local coarse shapes in $V$ such that for each $i$, there are infinitely many $F'$-classes in $Z(k)$ having local coarse shape in $\uparrow\!\gamma_i$. Now for each $i\in\omega$ let $C_i$ be an $F'$-class having local coarse shape $\gamma_{2i}$, and let $U=\cup_iC_i$. As in the previous paragraph, since $|X_{\alpha_k}/F|\leq |U/F'|$ and $\lcs(C)\leq\lcs(C')$ for every $F$-class $C\subseteq X_{\alpha_k}$ and $F'$-class $C'\subseteq U$, we may now define $g_k$ to embed $(E_1,F_1)\res X_{\alpha_k}$ into $(E_1',F_1')\res U$ using \ref{lem:adjust}. It only remains to check that $(\ast\ast)$ continues to hold. Let $W\in\,\uparrow\!\Sc$ be arbitrary. If there is an $F'$-class $C\subseteq Y_1$ such that $\lcs(C)\in W$ and $C\subseteq Z(k)\setminus Z(k+1)$, then by construction there is $i$ such that $\gamma_{2i}\in W$. But then for all $j\geq i$ there is an $F'$-class in $Z(k+1)$ with local coarse shape $\gamma_{2j+1}\in W$, so $\n_W\big((E_1',F_1')\res Z(k+1)\big)=\aleph_0$ as well. \end{proof}

\section{Classifying smooth countable pairs up to $\cong_B$}\label{sec:pairs2}

In this final section we determine the extent to which global fine shape is a complete invariant for simultaneous Borel isomorphism of pairs of smooth countable equivalence relations. We begin by identifying those functions that can arise as a global fine shape. It will be helpful to introduce some specialized notation. First, to any Borel set $B\subseteq X\times Y$ in a product of standard Borel spaces, we associate the ``section-counting" function $\hat{B}:Y\to\C$ defined by
\[
\hat{B}(y)\:=\:|B^y|.
\]
In particular, if $g:X\to Y$ is a Borel function then we think of $\hat{g}:Y\to\C$ as the fiber-counting function that associates to each $y\in Y$ the cardinality of $g^{-1}(y)$. Second, as in Section \ref{sec:singletons} we will continue to write $\Eq{g}$ for the smooth equivalence relation on $\dom(g)$ induced by the Borel function $g$.

\begin{lemma} Let $\phi:\Sf\to\C$ be an arbitrary function. Then $\phi$ is the global fine relative shape of some pair $E\subseteq F$ of smooth countable equivalence relations if and only if there is a standard Borel space $X$ and a Borel function $g:X\to\Sf$ such that $\phi=\hat{g}$. \label{lem:gs} \end{lemma}

\begin{proof} For the forward direction, suppose $\phi=\gfs(E,F)$, where $E\subseteq F$ are smooth countable equivalence relations on the standard Borel space $Y$. Let $X=Y/F$, and let $g([y]_F)=\lfs([y]_F)$. Then $\phi=\hat{g}$. For the converse, we will describe a general construction that canonically associates to any Borel function $g:X\to\Sf$ a pair of smooth equivalence relations $E_g(X)\subseteq F_g(X)$ such that $\hat{g}=\gfs(E_g(X),F_g(X))$.

Let $X$ be any fixed standard Borel space. Define the Borel set $U(X)\subseteq X\times\omega\times\omega\times\omega$ by
\[
(x,k,m,n)\:\in\:U(X)\quad\Leftrightarrow\quad n=0\:\vee\:k<n.
\]
Then on $U(X)$ define the equivalence relations $E_u(X)\subseteq F_u(X)$ by
\[
(x,k,m,n)\mathrel{F_u(X)}(x',k',m',n')\quad\Leftrightarrow\quad x=x'
\]
and
\[
\begin{array}{lll}
(x,k,m,n)\mathrel{E_u(X)}(x',k',m',n') & \Leftrightarrow & x=x'\:\wedge\: n=n'\:\wedge\:m=m'.
\end{array}
\]
If $X$ is clear from context or unimportant then we omit it from the notation and write simply $U$, $E_u$, and $F_u$.
The equivalence relations $E_u\subseteq F_u$ are smooth countable equivalence relations on $U$. Intuitively, $(x,k,m,n)$ is the $(k+1)$-th element of the $(m+1)$-th $E_u$-class of size $n$ inside $[(x,k,m,n)]_{F_u}$, where we interpret $n=0$ to mean of size $\omega$; in particular, each $F_u$-class has local fine shape $\langle\bar{\omega},\omega\rangle$.

Now, given the standard Borel space $X$ and any Borel function $g:X\to\Sf$, define the Borel set $B_g(X)\subseteq U(X)$ by
\[
\begin{array}{llllll}
(x,k,m,n)\in B_g(X)      & \Leftrightarrow &           & (\;n=0 & \wedge & m< g(x)(\omega)\;) \\
                         &                 & \mbox{or} & (\;k<n & \wedge & m< g(x)(n)\;).
                    \end{array}
\]
Then we define the equivalence relations
\[
E_g(X):=E_u(X)\res B_g(X)\quad\mbox{and}\quad F_g(X):=F_u(X)\res B_g(X),
\]
again omitting reference to $X$ when convenient. Then $E_g\subseteq F_g$ are smooth countable equivalence relations on $B_g$ such that for each element $(x,k,m,n)\in B_g$, the $F_g$-class of $(x,k,m,n)$ has local fine shape $g(x)$. From this one easily checks that $\gfs(E_g,F_g)=\hat{g}$. \end{proof}

It follows from Lemma \ref{lem:gs} that the global fine relative shape of a pair $E\subseteq F$ of smooth countable equivalence relations need not be a Borel function, though by \ref{fact:AnalyticRep} and \ref{fact:genunicity} it will be $\sigma(\bsig^1_1)$-measurable and in particular universally measurable, i.e., $\mu$-measurable for any $\sigma$-finite Borel measure $\mu$. We will use the construction of $E_g(X)\subseteq F_g(X)$ frequently throughout the remainder of this section, along with the facts that
\[
\hat{g}=\gfs(E_g,F_g)\quad\mbox{and}\quad\gfs(E,F)=\hat{\lfs}(E,F).
\]

Recall that for smooth countable equivalence relations $E\subseteq F$ on the standard Borel space $X$, the local fine relative shape function $\lfs^{(E,F)}:X/F\to\Sf$ associated to $(E,F)$ is Borel. In this section we often use the notation $\lfs(E,F)$ in place of $\lfs^{(E,F)}$.

\begin{lemma} Suppose that $E\subseteq F$ and $E'\subseteq F'$ are smooth countable equivalence relations. If $(E,F)\cong_B(E',F')$, then $\llbracket\lfs(E,F)\rrbracket\cong_B\llbracket\lfs(E',F')\rrbracket$.\label{lem:easy1} \end{lemma}

\begin{proof} Suppose $E\subseteq F$ are defined on $X$, and $E'\subseteq F'$ on $Y$. If $\phi:X\to Y$ is a simultaneous Borel isomorphism from $(E,F)$ to $(E',F')$, then the Borel function $\tilde{\phi}:X/F\to Y/F'$ defined by $\tilde{\phi}([x]_F)=[\phi(x)]_{F'}$ is an isomorphism from $\llbracket\lfs(E,F)\rrbracket$ to $\llbracket\lfs(E',F')\rrbracket$. \end{proof}

\begin{lemma} Let $g:X\to\Sf$ be any Borel function into $\Sf$. Then $\Eq{g}\mathrel{\cong_B}\Eq{\lfs(E_g,F_g)}$. In fact, there is a Borel isomorphism $\phi:X\to B_g/F_g$ such that for all $x\in X$, $g(x)=\lfs^{(E_g,F_g)}(\phi(x))$.\label{lem:easy2} \end{lemma}

\begin{proof} Recall that $\Eq{\lfs(E_g,F_g)}$ is an equivalence relation defined on $B_g(X)/F_g(X)$, where $B_g(X)\subseteq X\times\omega^3$. For each $x\in X$, let $\phi(x)$ be the unique $F_g$-class in $B_g$ whose elements have $x$ as their first component. Then $\phi$ is the desired isomorphism. \end{proof}

Using these observations together with Lemma \ref{lem:gs}, we can show that global fine shape is not a complete simultaneous isomorphism invariant, even amongst those smooth countable pairs with Borel global fine shape functions.

\begin{example} There exist smooth countable equivalence relations $E\subseteq F$ and $E'\subseteq F'$ such that $\gfs(E,F)=\gfs(E',F')$, but $(E,F)\not\cong_B(E',F')$. In fact, such pairs $(E,F)$ and $(E',F')$ can be found with $\gfs(E,F)$ constant. \label{ex:2} \end{example}

\begin{proof} Let $G$ and $H$ be smooth equivalence relations each of whose uncountably many equivalence classes is uncountable, and such that $G$ is \selective{} and $H$ is standard but not \selective{}, so in particular $G\not\cong_BH$. Then $G$ and $H$ admit surjective Borel reductions to $\Delta(\Sf)$, say $g$ and $h$ respectively. Note that $\hat{g}(\alpha)=\hat{h}(\alpha)=\cc$ for all $\alpha\in\Sf$. By Lemma \ref{lem:easy2}, we have
\[
\llbracket \lfs(E_{g},F_{g})\rrbracket \ \mathrel{\cong_B} \ G\quad\mbox{and}\quad\llbracket \lfs(E_{h},F_{h})\rrbracket \ \mathrel{\cong_B} \ H.
\]
Therefore
\[
\llbracket \lfs(E_{g},F_{g})\rrbracket \ \mathrel{\not\cong_B} \ \llbracket \lfs(E_{h},F_{h})\rrbracket,
\]
which by Lemma \ref{lem:easy1} implies that
\[
(E_{g},F_{g}) \ \mathrel{\not\cong_B} \ (E_{h},F_{h}),
\]
even though
\[\gfs(E_{g},F_{g})\quad=\quad\hat{g}\quad=\quad\hat{h}\quad=\quad\gfs(E_{h},F_{h}).\qedhere\]
\end{proof}

\begin{remark}\label{rem:MoreBS} Example \ref{ex:2} shows that one way to find smooth countable pairs $E\subseteq F$ and $E'\subseteq F'$ such that $\gfs(E,F)=\gfs(E',F')$ but $(E,F)\not\cong_B(E',F')$ is to find Borel functions $g,g'$ into $\Sf$ such that $\hat{g}=\hat{g}'$ but $\Eq{g}\not\cong_B\llbracket g'\rrbracket$. Indeed, we can produce families of pairwise non-isomorphic smooth countable pairs from families of pairwise non-isomorphic smooth singletons as follows. Suppose $\{D_i\st i\in I\}$ is a family of smooth equivalence relations such that:
\begin{enumerate}
 \item each $D_i$ has uncountably many equivalence classes, all of them uncountable;
 \item each $D_i$ is standard;
 \item for $i\ne j\in I$, $D_i\not\cong_BD_j$.
\end{enumerate}
For each $i\in I$, let $g_i$ be a surjective Borel reduction from $D_i$ to $\Delta(\Sf)$. Then
\[
\{(E_{g_i},F_{g_i})\st i\in I\}
 \]
is a pairwise non-isomorphic family of pairs of smooth countable equivalence relations all having the same (constant) global fine shape. \end{remark}


On the other hand, a smooth countable pair $E\subseteq F$ having global fine shape function constantly $\cc$ is rather unnatural. We turn our attention next to identifying a large natural class of smooth countable pairs on which global fine shape \emph{is} a complete isomorphism invariant. It turns out that again the notion of Borel parametrization plays a central role.

\begin{theorem} Let $E\subseteq F$ and $E'\subseteq F'$ be smooth countable equivalence relations on the standard Borel spaces $X$ and $Y$, respectively, such that $\llbracket \lfs(E,F)\rrbracket$ and $\llbracket \lfs(E',F')\rrbracket$ are Borel parametrized. Then $(E,F)\cong_B(E',F')$ if and only if $\gfs(E,F)=\gfs(E',F')$.\label{thm:SmoothPairsIsomBP} \end{theorem}

Before proving Theorem \ref{thm:SmoothPairsIsomBP} we make some remarks about its statement and establish a lemma. First, recall that the local fine shape function $\lfs(E,F)$ is defined not on $X$ but on $X/F$. Let us temporarily write $\dot{\lfs}{}(E,F)$ for the point map $x \mapsto\lfs(E,F)([x]_F)$ from $X$ to $\Sf$. Then $\dot{\lfs}{}(E,F)$ determines the smooth equivalence relation $\llbracket\dot{\lfs}{}(E,F)\rrbracket$ on $X$, and $\llbracket\lfs(E,F)\rrbracket$ is just the quotient of $\llbracket\dot{\lfs}{}(E,F)\rrbracket$ by $F$. We state Theorem \ref{thm:SmoothPairsIsomBP} using $\llbracket\lfs(E,F)\rrbracket$ rather than $\llbracket\dot{\lfs}{}(E,F)\rrbracket$ because it is the former that naturally arises in the proof. However, the following fact implies that $\llbracket\lfs(E,F)\rrbracket$ is Borel parametrized if and only if $\llbracket\dot{\lfs}{}(E,F)\rrbracket$ is, so either could be used in the statement of Theorem \ref{thm:SmoothPairsIsomBP}. The proof of Proposition \ref{lem:technicalBP} is given in the appendix (see \ref{prop:BPquotient}).

\begin{proposition} Let $D$ be a smooth equivalence relation on the standard Borel space $X$, and let $F\subseteq D$ be a smooth countable sub-equivalence relation. Then $D$ is Borel parametrized if and only if $D/F$ is Borel parametrized.\label{lem:technicalBP} \end{proposition}

The following lemma will be used in the proof of Theorem \ref{thm:SmoothPairsIsomBP}, and will imply that $\gfs(E,F)$ being Borel is a necessary (though by \ref{ex:2} not sufficient) condition for the isomorphism type of $(E,F)$ to be completely determined by $\gfs(E,F)$.

\begin{lemma} Let $X$ and $Y$ be standard Borel spaces and let $g:X\to Y$ be any Borel function. Then $\hat{g}$ is Borel if and only if $\Eq{g}$ splits and is standard. In particular, if $\Eq{g}$ is Borel parametrized then $\hat{g}$ is Borel.\label{lem:BPhat} \end{lemma}

\begin{proof} Suppose $\Eq{g}$ splits and is standard. Since $\Eq{g}$ is standard, $Y\setminus\ran(g)$ is Borel, so $\hat{g}^{-1}(0)$ is Borel. Since $\Eq{g}$ splits, the set $C=\{y\in Y\st g^{-1}(y)\mbox{ is uncountable}\}=\hat{g}^{-1}(\cc)$ is Borel. That $\hat{g}^{-1}(m)$ is Borel for each $1\leq m\leq\omega$ now follows from \ref{fact:ctblPartition}.  Conversely, if $\Eq{g}$ is not standard then $\hat{g}^{-1}(0)=Y\setminus\ran(g)$ is not Borel, and if $\Eq{g}$ does not split then
\[
\{(x,y)\in\mbox{graph}(g)\st g^{-1}(y)\mbox{ is uncountable}\} \ = \ \pi_Y^{-1}\big(\hat{g}^{-1}(\cc)\big)
\]
is not Borel, which implies that $\hat{g}^{-1}(\cc)$ is not Borel. \end{proof}

\begin{proof}[Proof of Theorem \ref{thm:SmoothPairsIsomBP}] By Proposition \ref{prop:shape2}, we need only establish the backward direction. Let
\[
\phi_X \ : \ X/F \ \to \ \bigsqcup_{\m\in\C^+}(Z_\m\times Y_\m)
\]
be a Borel parametrization of $\llbracket \lfs(E,F)\rrbracket$. For each $\m\in\C^+$, let
\[
Z_{\m}(X) \ = \ \gfs(E,F)^{-1}(\m) \ = \ \left(\hat{\lfs}(E,F)\right)^{-1}(\m),
\]
so that $Z_{\m}(X)$ is the set of all $\alpha\in\Sf$ such that there are exactly $\m$ $F$-classes $C$ having local fine shape $\alpha$. Then each $Z_{\m}(X)$ is Borel by Lemma \ref{lem:BPhat}, and the map
\[
\begin{array}{rrrcl}
\psi_X & : & \displaystyle{\bigsqcup_{\m\in\C^+}\left(Z_\m\times Y_\m\right)} & \rightarrow & \displaystyle{\bigsqcup_{\m\in\C^+}\left(Z_\m(X)\times Y_\m\right)} \vspace{2mm} \\
& & (z,y) & \mapsto & \left(\,\lfs^{(E,F)}\left(\phi_X^{-1}(z,y)\right),\;y\,\right)
\end{array}
\]
is a Borel isomorphism from
\[
\bigsqcup_{\m\in\C^+}\left(\Delta(Z_\m)\times I(Y_\m)\right)\quad\mbox{to}\quad\bigsqcup_{\m\in\C^+}\left(\Delta(Z_{\m}(X))\times I(Y_\m)\right)
\]
such that for every $F$-class $C\subseteq X$, 
\[
(\pi_1\circ\psi_X\circ \phi_X)(C) \ = \ \lfs^{(E,F)}(C).
\]
In an exactly analogous manner define $\phi_Y$ and $\psi_Y$ using $(E',F')$ in place of $(E,F)$. Since $\gfs(E,F)=\gfs(E',F')$ and the definitions of $Z_{\m}(X)$ and $Z_{\m}(Y)$ depended only on $\gfs(E,F)$ and $\gfs(E',F')$, respectively, we have $Z_{\m}(X)=Z_{\m}(Y)$ for each $\m$. Therefore
\[
\rho \ := \ \phi_Y^{-1}\circ\psi_Y^{-1}\circ\psi_X\circ\phi_X
\]
is a Borel bijection from $X/F$ to $Y/F'$ such that for every $F$-class $C\in X/F$, $\lfs^{(E,F)}(C)=\lfs^{(E',F')}(\rho(C))$. Now $\rho$ induces a Borel reduction from $F$ to $F'$ with the same property, so $(E,F)\cong_B(E',F')$ by Lemma \ref{lem:adjust}. \end{proof}

We have the following converse of Theorem \ref{thm:SmoothPairsIsomBP}, which should be compared to Theorem \ref{thm:BP}.

\begin{theorem} Let $E\subseteq F$ be smooth countable equivalence relations. If $\llbracket \lfs(E,F)\rrbracket$ is not Borel parametrized, then there exist smooth countable equivalence relations $E'\subseteq F'$ such that $\gfs(E,F)=\gfs(E',F')$ but $(E,F)\not\cong_B(E',F')$.\label{thm:SmoothPairsNotBP} \end{theorem}

In order to prove this we will need the following strengthening of Lemma \ref{lem:easy1}.

\begin{lemma} Let $E\subseteq F$ and $E'\subseteq F'$ be smooth countable equivalence relations. Then $(E,F)\cong_B(E',F')$ if and only if there exists a Borel isomorphism $\phi$ from $\llbracket \lfs(E,F) \rrbracket$ to $\llbracket \lfs(E',F') \rrbracket$ that preserves local fine shapes, i.e., such that for every $F$-class $C$, $\lfs(E,F)(C)=\lfs(E',F')(\phi(C))$.\label{lem:last} \end{lemma}

\begin{proof} Suppose $E\subseteq F$ are defined on $X$ and $E'\subseteq F'$ on $Y$. The forward direction is clear, so assume $\phi:X/F\to Y/F'$ is a Borel isomorphism from $\llbracket \lfs(E,F) \rrbracket$ to $\llbracket \lfs(E',F') \rrbracket$ that preserves local fine shapes. Define the smooth equivalence relation $D$ on $X$ by
\[
x\mathrel{D}y \quad \Leftrightarrow \quad x\mathrel{F}y \ \wedge \ |[x]_E|=|[y]_E|.
\]
Then $E\subseteq D\subseteq F$. Using \ref{fact:ctblSmooth}, well-order in a uniform Borel manner the collection of $E$-classes within each $D$-class in order type $\omega$ or finite, and for each $x\in X$ let $m(x)$ be the index of $[x]_E$ in this well-ordering. Likewise, well-order in a uniform Borel manner the elements of each $E$-class in order type $\omega$ or finite, and for each $x\in X$ let $k(x)$ be the index of $x$ in this well-ordering. Also for each $x\in X$ let $n(x)=0$ if $[x]_E$ is infinite, and let $n(x)=|[x]_E|$ otherwise. The functions $x\mapsto k(x),m(x),n(x)$ are all Borel. Define the injective Borel function $\Phi:X\to X/F\times\omega^3$ by
\[
\Phi(x)=\langle [x]_F,k(x),m(x),n(x)\rangle.
\]
In an exactly analogous manner, define $\Phi':Y\to Y/F'\times\omega^3$ by
\[
\Phi'(y)=\langle [y]_{F'},k(y),m(y),n(y)\rangle.
\]
Define $\rho:X/F\times\omega^3\to Y/F'\times\omega^3$ by $\rho(\langle C,k,m,n\rangle)=\langle \phi(C),k,m,n\rangle$. Then $\Phi'^{-1}\circ\rho\circ\Phi$ is a simultaneous Borel isomorphism from $(E,F)$ to $(E',F')$. \end{proof}

\begin{proof}[Proof of Theorem \ref{thm:SmoothPairsNotBP}] We consider two cases. First suppose that $\llbracket \lfs(E,F)\rrbracket$ splits and is standard, so that $\gfs(E,F)$ is Borel by Lemma \ref{lem:BPhat}. For each $\m\in\C^+$ let $Y_\m$ be a fixed Borel subset of $\mathcal{N}$ of cardinality $\m$, and let
\[
P \ := \ \bigsqcup_{\m\in\C^+}\gfs(E,F)^{-1}(\m)\times Y_{\m} \quad \subseteq \quad \Sf\times\mathcal{N}.
\]
Since $\gfs(E,F)$ is Borel, so is $P$. Let $g:P\to\Sf$ be the projection map onto the first coordinate, so that $\hat{g}=\gfs(E,F)$ and $\Eq{g}$ is Borel parametrized. By our observations following the proof of Lemma \ref{lem:gs}, $\hat{g}=\gfs(\,E_g(P),F_g(P))$. Furthermore, $\Eq{g}\cong_B\llbracket \,\lfs(E_g(P),F_g(P))\,\rrbracket$ by Lemma \ref{lem:easy2}, so that $\llbracket \,\lfs(E_g(P),F_g(P))\,\rrbracket$ is Borel parametrized and thus is not isomorphic to $\llbracket \lfs(E,F)\rrbracket$. Hence $(E_g(P),F_g(P))$ is a pair of smooth countable equivalence relations such that $(E,F)\not\cong_B(E_g(P),F_g(P))$ even though $\gfs(E,F)=\gfs(E_g(P),F_g(P))$.

Now suppose that $\llbracket \lfs(E,F)\rrbracket$ either is not standard or does not split. Let
\[
U=\{\alpha\in \Sf\st \gfs(E,F)(\alpha)=\cc\},
\]
so that $U$ is non-Borel analytic, and hence uncountable. Fix an uncountable Borel set $B\subseteq U$.  Let $\tilde{X}=X/F$, and let $G$ be the Borel set
\[
G \ := \ \mbox{graph}\,(\lfs(E,F)) \ \bigcap \ \tilde{X}\times (\Sf\setminus B) \quad \subseteq \quad \tilde{X}\times \Sf.
\]
Let $D$ be a Borel subset of $\tilde{X}\times B$ with all horizontal sections $D^x$ uncountable such that $D$ does not admit a Borel uniformization. Let $P=D\cup G$ and let $P'=(\tilde{X}\times B)\cup G$. Let $g:P\to\Sf$ and $g':P'\to\Sf$ be the projections onto $\Sf$. By construction,
\[
\begin{array}{lllll}
\gfs(E,F) & = & \hat{g} & = & \gfs(E_g(P),F_g(P)) \\
               & = & \hat{g}' & = & \gfs(E_{g'}(P'),F_{g'}(P')).
\end{array}
\]
We claim, however, that $(E_g,F_g)\not\cong_B(E_{g'},F_{g'})$, so that $(E,F)$ must fail to be isomorphic to at least one of them. To see this, suppose for contradiction that $(E_g,F_g)\cong_B(E_{g'},F_{g'})$. Then by Lemma \ref{lem:last} there is a Borel isomorphism $\phi$ from $\Eq{\lfs(E_g,F_g)}$ to $\Eq{\lfs(E_{g'},F_{g'})}$ that preserves local fine shapes. By Lemma \ref{lem:easy2}, this implies that there is a Borel isomorphism $\phi':P\to P'$ from $\Eq{g}$ to $\Eq{g'}$ such that for all $p\in P$, $g(p)=g'(\phi'(p))$. In particular, then, $\phi'$ is a Borel isomorphism from $\Eq{g}\res\tilde{X}\times B$ to $\Eq{g'}\res\tilde{X}\times B$, which is impossible since the latter is selective but the former is not.
\end{proof}

This leaves us with the analogue of Problem \ref{prob:1} for smooth countable pairs.

\begin{problem} Classify pairs $E\subseteq F$ of smooth countable equivalence relations (for which $\Eq{\lfs(E,F)}$ is not Borel parametrized) up to simultaneous Borel isomorphism. \label{prob:2} \end{problem}

We summarize some of the constructions of this section in Figure \ref{fig:1}, which we now explain. For the classes of smooth equivalence relations, pairs of smooth countable equivalence relations, and local fine shape functions (i.e., Borel functions into $\Sf$), we actually mean to consider equivalence classes of these objects under the obvious notions of equivalence. For the first two this is just $\cong_B$, and for Borel functions into $\Sf$ it is equality up to a Borel isomorphism of domains, or essentially what in Section \ref{sec:singletons} we called \emph{Borel equivalence} following \cite{KMM}. Then the solid vertical and horizontal arrows represent explicit constructions that respect these equivalences. The two solid diagonal arrows are just the compositions of the appropriate vertical and horizontal ones.

The content of the diagram of solid arrows lies in the fact that the vertical arrows are mutually inverse bijections, while the horizontal ones are surjective. So for any Borel function $g:X\to\Sf$ there is a Borel isomorphism $\pi:B_g(X)\to X$ such that
\[
g\pi \ = \ \lfs(E_g,F_g),
\]
and for any pair $E\subseteq F$ of smooth equivalence relations on the standard Borel space $X$ we have
\[
(E,F) \ \cong_B \ \left(\, E_{\mbox{\scriptsize{\texttt{lfs}{$(E,F)$}}}}(X/F),\;F_{\mbox{\scriptsize{\texttt{lfs}{$(E,F)$}}}}(X/F)\,\right).
\]
Hence the diagram immediately displays the facts that
\[
\hat{g} \ = \ \gfs\big(E_g(X),F_g(X)\big) \quad \mbox{and} \quad \Eq{g} \ \cong_B \ \llbracket\lfs(E_g,F_g)\rrbracket.
\]

Next consider the dotted arrows. For every smooth equivalence relation $D$ there exists a Borel reduction $g$ from $D$ to $\Delta(\Sf)$, and fixing such $g$ we obtain the pair $(E_g,F_g)$ and note that $D\cong_B\Eq{g}\cong_B\llbracket\lfs(E_g,F_g)\rrbracket$. Likewise any global fine shape function $\phi$ arises as $\hat{g}$ for some Borel function $g$ into $\Sf$, in which case $\phi=\hat{g}=\gfs(E_g,F_g)$. However, we have no canonical way of recovering $g$ from either $\Eq{g}$ or $\hat{g}$ in general, which explains the use of dotted arrows.

\begin{figure}
\begin{center}
\begin{tikzpicture}[scale=1.0]
 \draw (-1.4,-.5) rectangle (1.4,1); 
 \node at (0,.5) {local fine};
 \node at (0,0) {shapes}; 
 \draw (3.4,-.5) rectangle (6.2,1);
 \node at (4.8,.5) {global fine};
 \node at (4.8,0) {shapes};
 \draw (-6.2,-.5) rectangle (-3.4,1);
 \node at (-4.8,.5) {smooth eq.};
 \node at (-4.8,0) {relations};
 \draw (-1.4,3.5) rectangle (1.4,5);
 \node at (0,4.5) {smooth};
 \node at (0,4) {countable pairs};
 \draw [->>] (1.6,.45) -- (3.2,.45);
 \draw [->>] (-1.6,.45) -- (-3.2,.45);
 \draw [<-] [dotted] (1.6,.15) -- (3.2,.15);
 \draw [<-] [dotted] (-1.6,.15) -- (-3.2,.15);
 \draw [>->>] (.4,1.2) -- (.4,3.3);
 \draw [<<-<] (-.4,1.2) -- (-.4,3.3);
 \draw [<-] [dotted] (1.6,4) -- (4.6,1.2);
 \draw [<-] [dotted] (-1.6,4) -- (-4.6,1.2);
 \node at (1.6,2.25) {$g\mapsto (E_g,F_g)$};
 \node at (-1.2,2.25) {$\lfs(\,\cdot\,,\,\cdot\,)$};
 \node at (2.4,.7) {$g\mapsto\hat{g}$};
 \node at (-2.4,.75) {$\Eq{g}\mapsfrom g$}; 
 \draw [<<-] (5,1.2) -- (1.6,4.4);
 \node at (5,3) {$\gfs(\,\cdot\,,\,\cdot\,)={\hat{\rule[2.5mm]{0mm}{0mm}\lfs}}{}(\,\cdot\,,\,\cdot\,)$};
 \draw [<<-] (-5,1.2) -- (-1.6,4.4);
 \node at (-4.2,3) {$\llbracket\lfs(\,\cdot\,,\,\cdot\,)\rrbracket$};
\end{tikzpicture}
\end{center}
\caption{}\label{fig:1}
\end{figure}

Finally, to conclude this section we consider the question of relative complexities of classification problems. Let us slightly relax the notion of Borel equivalence from \cite{KMM} that was discussed in Section \ref{sec:singletons} to allow $f$ and $g$ to have different domains. Thus we say that Borel functions $f:X\to\Sf$ and $g:Y\to\Sf$ are \emph{Borel equivalent}, and write $f\equiv_Bg$, if there is a Borel bijection $\phi:X\to Y$ such that $f=g\phi$. Let us also say that $f,g$ are \emph{weakly Borel equivalent}, and write $f\equiv_B^wg$, if there exist Borel bijections $\phi:X\to Y$ and $\psi:\ran(f)\to\ran(g)$ such that $\psi f=g\phi$. (Here we ask $\psi$ to be Borel measurable, which makes sense even if $\ran(f)$ and $\ran(g)$ are non-Borel analytic). So we have $\equiv_B\;\subseteq\;\equiv_B^w$, and it is easy to check that for any Borel functions $f:X\to\Sf$ and $g:Y\to\Sf$, $f\equiv_B^wg\:\Leftrightarrow\:\Eq{f}\cong_B\Eq{g}$.

Unfortunately, there is no natural way to realize the collections of smooth equivalence relations, pairs of smooth countable equivalence relations, or Borel functions into $\Sf$ as standard Borel spaces, so the usual framework of Borel equivalence relations does not apply to the equivalences on these classes pictured in Figure \ref{fig:1}. Nevertheless, we can interpret the correspondences given by the vertical arrows as expressing the fact that the problem of classifying pairs of smooth countable equivalence relations up to $\cong_B$ is essentially identical to that of classifying Borel functions into $\Sf$ up to $\equiv_B$. Similarly, the assignment $g\mapsto\Eq{g}$ can be understood as reducing the problem of classifying Borel functions up to $\equiv_B^w$ to that of classifying smooth equivalence relations up to $\cong_B$.

On the other hand, we know of no canonical way of recovering from $E$ a Borel function $g$ such that $E\cong_B\Eq{g}$, and likewise no canonical way of choosing a $\,\equiv_B$-class from within a $\,\equiv_B^w$-class. 
It would be interesting to find natural reductions, if they exist, relating the problems of Borel equivalence and weak Borel equivalence of Borel functions, or equivalently to determine whether there exist natural mappings in either direction between smooth equivalence relations and smooth countable pairs of equivalence relations that reduce the isomorphism problem of one to that of the other.

\appendix


\section*{Appendix}\label{sec:appendix}

We collect here some basic facts from descriptive set theory that are needed in the main body of the paper, along with the proof of Proposition \ref{lem:technicalBP}.

\section{Some background from Descriptive Set Theory}\label{sec:App1}

A \emph{standard Borel space} is a measurable space $(X,\mathcal{B})$ such that $\mathcal{B}$ is the $\sigma$-algebra of Borel sets generated by some Polish topology on $X$. Here a topological space $(X,\tau)$ is \emph{Polish} if it is separable and there is a complete metric on $X$ compatible with $\tau$. Our standard example is Baire space $\mathcal{N}=\omega^\omega$ of sequences of natural numbers, which is Polish in the product of discrete topologies. A map $f:X\to Y$ between standard Borel spaces $(X,\mathcal{B}_X)$ and $(Y,\mathcal{B}_Y)$ is \emph{Borel} if $f^{-1}(B)\in\mathcal{B}_X$ for every $B\in\mathcal{B}_Y$, or equivalently if $\mbox{graph}(f)$ is a Borel subset of the product $X\times Y$; $f$ is \emph{bimeasurable} if additionally $f(B)\in\mathcal{B}_Y$ for every $B\in\mathcal{B}_X$. A bimeasurable bijection between standard Borel spaces is called an \emph{isomorphism}. The image of a Borel set under an injective Borel function is Borel, so that an injective Borel function is an isomorphism onto its range. We make constant use of the following fact, due to Kuratowski:

\begin{fact}[The isomorphism theorem] The standard Borel spaces $X$ and $Y$ are isomorphic if and only if they have the same cardinality.\label{fact:kuratowski} \end{fact}

A set $A$ in a standard Borel space $X$ is \emph{analytic} if it is the image of some Borel set under a Borel function, and \emph{coanalytic} if its complement is analytic. Suslin showed that non-Borel analytic sets exist, and further proved the following:

\begin{fact}[Suslin's Theorem] A subset of a standard Borel space is Borel if and only if it is both analytic and coanalytic.\label{fact:suslin} \end{fact}

Concerning the possible cardinalities of analytic and coanalytic sets, we have:

\begin{fact} Every uncountable analytic (in particular Borel) set contains a homeomorph of $2^\omega$ and therefore has cardinality $\cc$ (\cite[4.3.5]{Srivastava}). Under the assumption of Analytic Determinacy, the same is true of coanalytic sets; in ZFC, every uncountable coanalytic set has cardinality $\aleph_1$ or $\cc$ (\cite[4.3.17]{Srivastava}).\label{fact:card} \end{fact}

From \ref{fact:kuratowski} and \ref{fact:card} it follows that there is exactly one uncountable standard Borel space up to isomorphism. Furthermore, since sections of Borel sets are Borel, from \ref{fact:card} we have that every equivalence class of a Borel equivalence relation has cardinality in $\C^+=\{1,2,\ldots,\aleph_0,\cc\}$. A powerful generalization of \ref{fact:card} that we use constantly is Silver's Theorem.

\begin{fact}[Silver's Theorem \cite{Silver}] Let $E$ be a coanalytic equivalence relation on the standard Borel space $X$. If $E$ has uncountably many classes, then $\Delta(2^\omega)\leq_BE$.\label{fact:Silver} \end{fact}

We require a few additional facts concerning analytic sets, coanalytic sets, and cardinality.

\begin{fact}[{\cite[4.3.7]{Srivastava}, \cite[29.19 and 29.21]{Kechris}}] \label{fact:AnalyticRep} Let $X$ and $Y$ be standard Borel spaces. If $B\subseteq X\times Y$ is Borel (or even analytic), then $\{x\in X\st B_x\mbox{ is uncountable}\}$ is analytic.  In fact, a set $A\subseteq X$ is analytic if and only if there exists a Borel set $B\subseteq X\times\mathcal{N}$ such that
\[
A = \{x\in X\st B_x\mbox{ is uncountable}\} = \{x\in X\st B_x\ne\emptyset\}.
\]
\end{fact}

On the other hand, Lusin's Unicity Theorem \cite[18.11]{Kechris} states that for $B\subseteq X\times Y$ Borel, the set $\{x\in X\::\: |B_x|=1\}$ is coanalytic. We need the following generalization of this:

\begin{fact}[{\cite[Lemma 1]{KMM}}] Let $X$ and $Y$ be standard Borel spaces, with $B\subseteq X\times Y$ Borel. Then for each $n=0,1,\ldots,\aleph_0$, the set $\{x\in X\::\: |B_x|=n\}$ is coanalytic. \label{fact:genunicity} \end{fact}

Using these facts together with Silver's Theorem we can prove the following.

\begin{proposition} Let $E$ be a Borel equivalence relation on the standard Borel space $X$. Then
\begin{enumerate}
\item[(i)] for all $\m\in\C^+$, $\n_{\m}(E)\in\C\cup\{\aleph_1\}$;
\item[(ii)] $\n_\cc(E)\in\C$;
\item[(iii)] if $X^{(E)}_\cc$ is Borel (i.e., if $E$ splits), then $X^{(E)}_{\m}$ is Borel and $\n_{\m}(E)\in\C$ for all $\m\in\C^+$.
\end{enumerate} \label{prop:card}
\end{proposition}

\begin{proof}  Let $E$ be a Borel equivalence relation on the standard Borel space $X$. Recall that for each $\m\in\C$, $X^{(E)}_{\m}$ is the set of $x\in X$ such that $|[x]_E|=\m$, so that we have $\n_{\m}(E)\cdot\m=|X^{(E)}_{\m}|$. Since $X^{(E)}_{\leq\omega}$ is coanalytic, $E\cup\big(X^{(E)}_{\leq\omega}\times X^{(E)}_{\leq\omega}\big)$ is a coanalytic equivalence relation on $X$ with $\n_\cc(E)+1$ many classes; therefore if $\n_\cc(E)$ is uncountable, then we must have $\n_\cc(E)=\cc$ by Silver's Theorem. This proves (ii) and part of (i). To complete the proof of (i), suppose that $\m\in\C^+$ is countable, so that $X^{(E)}_\m$ is coanalytic by \ref{fact:genunicity}. Then if $\n_{\m}(E)$ is uncountable, we must have $\n_{\m}(E)=|X^{(E)}_\m|\in\{\aleph_1,\cc\}$ by \ref{fact:card}. Finally, if $X^{(E)}_\cc$ is Borel then $E\res X^{(E)}_{\leq\omega}$ is a countable Borel equivalence relation and claim (iii) is Fact \ref{fact:ctblPartition} below.
\end{proof}

In ZFC we cannot rule out $\aleph_1$ as a possible cardinality for $\n_\m(E)$, $\m\ne\cc$.

\begin{proposition}\label{prop:pathological} It is consistent with ZFC$\,+\,\neg$CH that for each function $\alpha:\C^+\to\C\cup\{\aleph_1\}$ such that $\alpha(\cc)=\cc$, there is a smooth equivalence relation $E$ such that $\fs(E)=\alpha$. \end{proposition}

\begin{proof} Decompose $\mathcal{N}$ into countably many homeomorphic copies of itself, $\mathcal{N}=\bigsqcup Y_m$, where $m$ ranges over $\{0,1,2,\ldots,\aleph_0\}$. Let $\alpha$ be given as in the hypothesis. Working in a model of ZFC$\,+\,\neg$CH, for each $m\in\{1,2,\ldots,\aleph_0\}$ let $C_m$ be a coanalytic set of cardinality $\alpha(m)$ in $Y_m$. By \cite[Theorem 5]{KMM} there is a Borel function $g:\mathcal{N}\to\mathcal{N}$ such that for each $m\in\{1,\ldots,\aleph_0\}$, $C_m=\{y\in\mathcal{N}\st |g^{-1}(\{y\})|=m\}$, and by adjusting $g$ if necessary we can take it to have uncountably many uncountable fibers. Now let $x\mathrel{E}y\Leftrightarrow g(x)=g(y)$. \end{proof}

Next we recall the Lusin-Novikov uniformization theorem together with some of its consequences. If $B\subseteq X\times Y$, a \emph{uniformization} of $B$ is a subset $C\subseteq B$ such that $\pi_X(B)=\pi_X(C)$ and for every $x\in X$, $C_x$ contains at most one point. Since images of Borel sets under injective Borel functions are Borel, any set $B\subseteq X\times Y$ admitting a Borel uniformization has Borel projection onto $X$. 

\begin{fact}[Lusin-Novikov uniformization] Let $X$ and $Y$ be standard Borel spaces and let $B\subseteq X\times Y$ be Borel. If every section $B_x$ of $B$ is countable, then $B$ admits a Borel uniformization and therefore $\pi_X(B)$ is Borel. 
(See \cite[5.8.11]{Srivastava} or \cite[18.10]{Kechris}). \end{fact}

The Lusin-Novikov uniformization theorem can be used to prove an important representation theorem for countable Borel equivalence relations due to Feldman and Moore (see \cite[5.8.13]{Srivastava}).

\begin{fact}[Feldman-Moore, \cite{FM}] If $E$ is a countable Borel equivalence relation on the standard Borel space $X$, then there is a countable group $\Gamma$ and a Borel action of $\Gamma$ on $X$ such that $E$ is the orbit equivalence relation $E^X_\Gamma$ arising from the action.\label{fact:FM} \end{fact}

\begin{fact}[see {\cite[18.15]{Kechris}}] If $E$ is a countable Borel equivalence relation on the standard Borel space $X$, then for each $1\leq m\leq\omega$, the set $X_m=\{x\in X\st |[x]_E|=m\}$ is Borel.\label{fact:ctblPartition} \end{fact}

\begin{proof} Fix by the Feldman-Moore theorem a countable group $\Gamma=\{\gamma_i\st i\in\omega\}$ and a Borel action of $\Gamma$ on $X$ such that $E=E^X_\Gamma$. The claim follows from the fact that for each $1\leq m<\omega$,
\[
x\in X_{\geq m} \quad \Leftrightarrow \quad (\exists i_1,\ldots,i_m\in\omega)\bigwedge_{1\leq j\ne k\leq m}\gamma_{i_j}\cdot x\ne\gamma_{i_k}\cdot x
\]
and
\[
x\in X_{\leq m} \quad \Leftrightarrow \quad (\forall i_0,\ldots,i_m\in\omega)\bigvee_{0\leq j\ne k\leq m}\gamma_{i_j}\cdot x=\gamma_{i_k}\cdot x. \qedhere
\]
\end{proof}

Next we recall examples that can be used to separate the notions of smooth, standard, selective, and split Borel equivalence relations.

\begin{example}[{\cite[26.2\mbox{ and }29.21]{Kechris}}] For any uncountable standard Borel spaces $X$ and $Y$, there exists a Borel set $P\subseteq X\times Y$ such that $\pi_X(P)$ is not Borel, so in particular $P$ does not admit a Borel uniformization. Moreover, $P$ can be taken to have all nonempty sections $P_x$ uncountable.\label{ex:AnalyticProjection} \end{example}

\begin{example}[{\cite[18.17]{Kechris}, \cite[5.1.7]{Srivastava}}] For any uncountable standard Borel spaces $X$ and $Y$, there exists a Borel set $P\subseteq X\times Y$ such that $\pi_X(P)=X$ but $P$ does not admit a Borel uniformization. Moreover, $P$ can be taken to have all sections $P_x$ uncountable.\label{ex:NotUniformizable} \end{example}

The following ``fattening" trick is sometimes helpful in obtaining uncountable sections in contexts similar to those of Examples \ref{ex:AnalyticProjection} and \ref{ex:NotUniformizable}.

\begin{fact} Let $X$ and $Y$ be standard Borel spaces and let $B\subseteq X\times Y$ be Borel. Then there exists a Borel set $D\subseteq X\times Y$ containing $B$ such that $\pi_X(B)=\pi_X(D)$, $D_x$ is uncountable for all $x\in \pi_X(D)$, and $B$ admits a Borel uniformization if and only if $D$ does. \end{fact}

\begin{proof} Without loss of generality, $X=Y=\mathcal{N}$. Fix a Borel bijection $\phi:\mathcal{N}\rightarrow\mathcal{N}^2$, and write $\phi(x)=(\phi_1(x),\phi_2(x))\in\mathcal{N}^2$. Define a new Borel set $B'\subseteq X\times Y$ by
\[
(x,y)\in B'\quad\Leftrightarrow\quad (x,\phi_1(y))\in B,
\]
and let $D=B\cup B'$. Clearly $D$ has the same projection as $B$ and has all nonempty sections uncountable. A Borel uniformization for $B$ is again one for $D$. Conversely, suppose $C$ is a Borel uniformization of $D$, and define $\alpha:X^2\to X^2$ by $\alpha(x,y)=(x,\phi_1(y))$. Then $\alpha(C\cap(B'\setminus B))\cup(C\cap B)$ is a Borel uniformization of $B$. \end{proof}

Finally, we record a useful theorem of Mauldin that we use to prove that coarse shape is a complete biembeddability invariant for smooth Borel equivalence relations (Proposition \ref{prop:smoothemb}).

\begin{fact}[Mauldin \cite{MauldinBF}] Let $X$ and $Y$ be Polish spaces, $A\subseteq X\times Y$ analytic, and suppose
\[
U:=\{x\in X\st A_x\mbox{ is uncountable}\}
\]
is uncountable. Then there is a nonempty compact perfect set $P\subseteq U$ and a Borel isomorphism $\phi$ of $P\times 2^\omega$ onto a subset $R$ of $A$ such that for each $p\in P$, $\phi\res \{p\}\times 2^\omega$ is a homeomorphism onto $\{p\}\times R_p$. \label{prop:Mauldin} \end{fact}

\section{Proof of Proposition \ref{lem:technicalBP}}\label{sec:App3}

\begin{proposition}\label{prop:BPquotient} Let $D$ be a smooth Borel equivalence relation on the standard Borel space $X$, and let $F\subseteq D$ be a smooth countable Borel sub-equivalence relation. Then $D$ is Borel parametrized if and only if $D/F$ is Borel parametrized. \end{proposition}

\begin{proof} Write $\tilde{X}=X/F$. Note that $X^{(D)}_{\leq\omega}$ is Borel if and only if $\tilde{X}^{(D/F)}_{\leq\omega}$ is Borel, so by \ref{cor:ctblBP} we may assume without loss of generality that every $D$-class is uncountable. Also by \ref{cor:ctblBP}, we may assume $|X/D|$ is uncountable. Furthermore, if either $D$ or $D/F$ is Borel parametrized, then $D$ is standard, so for the rest of the proof we assume $D$ is standard. Using \ref{prop:standard}, fix a standard Borel space $Y$ and a surjective Borel reduction $g:X\to Y$ from $D$ to $\Delta(Y)$. Let $G=\mbox{graph}(g)\subseteq X\times Y$, and let $D_0$ be the horizontal section equivalence relation on $G$, so that $D\cong_BD_0$. Also let $F_0$ denote the natural isomorphic copy of $F$ contained in $D_0$, namely $(x,y)\mathrel{F_0}(x',y')\Leftrightarrow x\mathrel{F}x'\Leftrightarrow x\mathrel{F}x'\wedge y=y'$.

Now, assume $D$ is Borel parametrized. Then by \cite[2.4]{MauldinBP} there is a conditional probability distribution $\mu$ on $Y\times\mathcal{B}_X$ such that for all $y\in Y$, $\mu(y,\cdot)$ is atomless and $\mu(y,G^y)=1$. Using \ref{fact:ctblSmooth}, let $\{T_n\st n\in\omega\}$ be a countable family of Borel transversals for $F_0$ such that $\cup_nT_n=G$. Then for each $y\in Y$, $G^y=\cup_nT_n^y$, so there must exist $n$ (depending on $y$) such that $\mu(y,T_n^y)>0$. For each $n$, let $Y_n$ be the set of $y\in Y$ for which $n$ is least such that $\mu(y,T_n^y)>0$. Define the Borel subset $T\subseteq G$ by
\[
(x,y)\in T\quad\Leftrightarrow\quad y\in Y_n\wedge (x,y)\in T_n.
\]
Then $T$ is a Borel transversal for $F_0$ such that for all $y\in Y$, $\mu(y,T^y)>0$. By \cite[2.1]{MauldinBP}, the function $\nu:Y\times\mathcal{B}_X\to\mathbb R$ defined by
\[
\nu(y,E)=\mu(y,T^y\cap E)
\]
is a conditional measure distribution on $Y\times\mathcal{B}_X$. For each $y\in Y$, $\nu(y,\cdot)$ is atomless and nontrivial since $\mu(y,\cdot)$ is, and $\nu(y,T^y)=\mu(y,T^y)>0$. It now follows from \cite[2.3]{MauldinBP} that $D_0\res T$ is Borel parametrized. But $D_0\res T\cong_B D_0/F_0$ via the quotient map, and $D_0/F_0\cong_B D/F$.

The proof of the converse is similar. Supposing that $D/F$ is Borel parametrized, fix a Borel transversal $T$ for $F_0$ and identify $D_0/F_0$ with $D_0\res T$ via the quotient map, so that $D_0\res T$ is Borel parametrized. By \cite[2.4]{MauldinBP} there is a conditional probability distribution $\nu$ on $Y\times\mathcal{B}_X$ such that for all $y\in Y$, $\nu(y,\cdot)$ is atomless and $\nu(y,T^y)=1$. Define $\mu:Y\times\mathcal{B}_X\to\mathbb R$ by
\[
\mu(y,E)=\nu(y,E\cap T^y).
\]
Then $\mu$ is a conditional probability distribution on $Y\times\mathcal{B}_X$ such that for each $y\in Y$, $\mu(y,\cdot)$ is atomless and $\mu(y,G^y)=1>0$. By \cite[2.4]{MauldinBP} it follows that $D_0$, and hence $D$, is Borel parametrized. \end{proof}

\end{document}